\documentclass[reqno,12pt]{article}
\usepackage{latexsym}
\usepackage{amsmath}
\usepackage{amssymb}
\usepackage{bbm}
\usepackage{amsthm}
\usepackage{color}
\usepackage{dsfont}
\usepackage{mathtools}

\usepackage{graphicx}
\usepackage{epsfig}
\usepackage{catchfilebetweentags}
\usepackage[
backend=biber,
style=numeric,
]{biblatex}
\definecolor{Red}{rgb}{1,0,0}

 0 \DeclareMathAlphabet{\mathbfit}{OT1}{cmr}{bx}{it}

\newcommand{\bea}{\begin{eqnarray*}}
\newcommand{\eea}{\end{eqnarray*}}
\newcommand{\be}{\begin{eqnarray}}
\newcommand{\ee}{\end{eqnarray}}
\newcommand{\beq}{\begin{equation}}
\newcommand{\eeq}{\end{equation}}

\newcommand{\zd}{\mathbb{Z}^d}

\newcommand{\confs}{\tilde{\Omega}}
\newcommand{\Q}{\mathcal{Q}}
\newcommand{\Zl}{Z_{\lambda}}
\newcommand{\Wl}{W_{\lambda}}
\renewcommand{\L}{\mathcal{L}}
\renewcommand{\P}{\mathbb{P}}
\newcommand{\E}{\mathbb{E}}
\newcommand{\N}{\mathbb{N}}

\newtheoremstyle{theorem}
                {10pt}
                {0pt}
                {}
                {}
                {\bfseries}
                {}
                {\newline}
                {}
\theoremstyle{theorem}
\newtheorem{thm}{Theorem}
\newtheorem{cor}[thm]{Corollary}
\newtheorem{rmk}[thm]{Remark}
\newtheorem{lem}[thm]{Lemma}

\newtheorem{defi}[thm]{Definition}

\begin{document}

\title{The speed of biased random walks among dynamical random conductances} 

\author{Eszter Couillard \thanks{Department of Mathematics, Technical University of Munich, \texttt{eszter.couillard@tum.de}}}
\date{}
\maketitle

\begin{abstract}
    We study biased variable-speed random walks in dynamical random conductances. Assuming that the conductances are upper-bounded, we prove that the walk has strictly positive speed for every bias $\lambda>0$. We then give an explicit asymptotic formula for the speed for $\lambda \to + \infty$, and prove two monotonicity properties for the speed. Finally, we provide an example showing that, even for conductances that are bounded and bounded away from zero, the speed can be asymptotically decreasing in the bias.
\end{abstract}

\section{Introduction}
In this paper, we discuss properties of the speed of variable speed biased random walks on dynamical conductances on $\mathbb{Z}^d$. Dynamical conductances $(\omega_t)_{t \geq 0}$ on $\mathbb{Z}^d$ assign to each time $t \geq 0$ and every edge $e \in E_d$ of $\mathbb{Z}^d$ a non-negative number $\omega_t(e)$. The distribution of $(\omega_t)_{t \geq 0}$ depends on two parameters: a measure $q$ on $[0,\infty)$ and $\mu > 0$. The process $(\omega_t)_{t \geq 0}$ is a Markov process with $\omega_0$ distributed according to $Q = q^{E_d}$, and the edges update their conductances independently at the points of a Poisson process with rate $\mu \in (0,\infty)$ and according to the measure $q$.\\

We study biased random walks in continuous time on $(\omega_t)_{t \geq 0}$ such that the time spent at a site depends on the conductances on the edges. We call these processes variable speed biased random walk (VBRW) (see Definition in Subsection \ref{DefVBRW}) and normalized variable speed biased random walk (NVBRW) (see Definition \ref{defNVBRW}). At jump times the random walker jumps according to probabilities that are proportional to $(\omega^{\lambda}_t)_{t \geq 0}$, where $\omega^{\lambda}_t(x,y) = \omega_t(\{x,y\}) e^{\lambda e \cdot e_1}$ for $x,y \in \mathbb{Z}^d$ with $x \sim y$, for some $\lambda > 0$ called the bias. Similar processes have been studied on dynamical percolation, unbiased random walks have be studied for instance in \cite{unbiasedPercolation}, and the biased case has been studied in \cite{percolation}.\\

For the (N-)VBRW, under the assumption that $q$ has a bounded support, we show the positivity of the speed (Theorem \ref{positiveSpeed}), then provide an asymptotic expansion in $\lambda$ for the speed (Theorem \ref{assymptotics}, Corollary \ref{corAssymtoticNVBRW}), and finally study the monotonicity of the speed as a function of $\lambda$ (Theorem \ref{Monotonmularge}, Theorem \ref{assymptoticMonotonicity}). The results obtained for the NVBRW can be seen as a generalization of those shown in \cite{percolation} for dynamical environments, with upper bounded conductances instead of percolation. \bigskip\\

The proof of the positivity of the speed requires two results: the CLT under the annealed probability measure and a relation between the probability to be at a point $x$ and the probability to be at point $-x$ at fixed time $t$. The relation between these two probabilities is obtained using a reversibility argument.\\
Our results show that an asymptotically decreasing speed does not only occur on dynamical percolation as shown in \cite{percolation}, but can also happen for measures $q$ that are uniformly elliptic. An explicit example is given.

\section{Definition of the model}
\subsection{Variable speed biased random walk (VBRW) in dynamical conductances} \label{DefVBRW}
Let $d\geq 1$. For $x,y \in \zd$ with $x\sim y$ and $t\geq0$ we denote by $\omega_t(x,y) = \omega_t(y,x)$ the non-negative conductance of the edge between $x$ and $y$ at time $t$. 
We may sometimes write $\omega_t(e)$ for $e \in E_d$ with $E_d$ the set of edges in $\zd$ and $\omega_t$ for the function associating the conductances to each edge at time $t\geq 0$.\\ 
Let $\kappa > 0$ and $q$ be a probability measure on $[0, \kappa]$, we assume $q \neq \delta_0$. For all $e \in E_d$ the conductances $\omega_0(e)$ are i.i.d. with law $q$, so $\omega_0$ is distributed according to the product measure $\Q:= q^{E_d}$. We will denote by $\confs$ the set $[0,\kappa]^{E_d}$.\\ 
Each edge updates its conductance independently of the other edges at the points of Poisson process with rate $\mu \in (0,\infty)$ and according to the measure $q$. The process $(\omega_t)_{t\geq 0}$ is then a Markov process with the product measure $\Q$ as its stationary distribution.\bigskip \\
We now want to define a biased random walk in the dynamical environment given by $(\omega_t)_{t \geq 0}$. To do so we first define the rates $\omega_t^{\lambda}$ for a bias $\lambda \geq 0$ and $t \geq 0$ in the following way 
\[
    \omega_t^{\lambda}(x,y) = e^{\lambda (y-x) \cdot e_1} \omega_t(x,y) \text{ for all } x,y\in \zd \text{ with } x \sim y
\]
with $e_1,...,e_d$ denoting the $d$ vectors of the standard orthonormal basis on $\zd$. Note that for $\lambda > 0 $ these rates are not symmetric.\\
We now construct $(X_t)_{t\geq 0}$, the variable speed biased random walker on $(\omega_t)_{t\geq 0}$. Let $\Wl(x,t)$ be the total jump rate at point $x \in \zd$ at time $t$

\[
\Wl(x,t):= \sum_{y \sim x} \omega_t^{\lambda}(x,y),
\]
then take $\mathcal{P}$ to be a Poisson process on with intensity function $\Wl(X_t,t)$. Then for $t \in \mathcal{P}$ the random walker $X$ jumps to a neighboring site $y$ with probability $\frac{\omega_t^{\lambda}(X_{t^-},y)}{\Wl(X_{t^-},t)}$. If not stated otherwise we will assume that $X_0 = 0$ a.s.

We will use $P^{\lambda}_{\omega}$ and $E^{\lambda}_\omega$ to denote the probability measure and the expectation corresponding to the random walk on the environment $\omega = (\omega_t)_{t\geq 0}$. We will write $\P^{\lambda}$ for the distribution of the joint process and $\E^{\lambda}$ for the corresponding expectation. If it is clear from the context what $\lambda$ is, we will omit it.\\
Note that the process $(\omega_t, X_t)_{t \geq 0}$ for $(X_t)_{t\geq 0}$ a biased random walk on $(\omega_t)_{t \geq 0}$ is a Markov process under $\P$. However, the process $(X_t)_{t\geq 0}$ alone is not a Markov process under $\P$.\bigskip\\
Let $(Z_t,\omega_t)_{t \geq 0}$ be a random walk on dynamical conductances.
If $\lim_{t \rightarrow \infty} \frac{X_t \cdot e_1}{t}$ exists $\P$-a.s. then we call it the (linear) speed of the random walker $(X_t)_{t\geq 0}$.\\
We will denote by $v(\lambda,\mu)$ the speed of a VBRW with parameters $\lambda, \mu>0$, in Lemma \ref{regenerativSpeed} we will show that this speed exists. Using the same arguments as for Lemma \ref{regenerativSpeed} and symmetry, we get that $\lim_{t \rightarrow \infty} \frac{X_t \cdot e_i}{t} = 0$ $\P$-a.s. for $i \in \{2,...,d\}$, so we are only interested in the speed in direction $e_1$.\\

\subsection{An alternative construction of the VBRW} \label{alternativeRep}
The goal of this section is to give an alternative description of VBRW on dynamical conductances.\\ 
Let $(\omega_t)_{t \geq 0}$ be dynamical conductances on $\zd$ with refreshing rate $\mu >0$ and measure $q$ on $[0,\kappa]$, and let $\lambda >0$. Let $\mathcal{P}$ be a Poisson point process with rate $Z_{\lambda} \kappa$ where $\Zl = 2d-2+ e^{\lambda} + e^{-\lambda}$. Then we construct the process $(X_t)_{t\geq 0}$ such that $X_0 = 0 $ and for $t \in \mathcal{P}$ we sample independently two independent random variables $U_t \sim Unif[0,Z_{\lambda}]$ and $V_t\sim Unif[0,\kappa]$, then 
\begin{enumerate}
    \item if $U_t \in [i-2,i-1)$, then $X$ attempts at time $t$ a jump in direction $e_i$ for $i \in \{2,...,d\}$
    \item if $U_t \in [d+i-3,d+i-2)$, then $X$ attempts at time $t$ a jump in direction $-e_i$ for $i \in \{2,...,d\}$
    \item if $U_t \in [2d-2,2d-2+e^{\lambda})$, then $X$ attempts at time $t$ a jump in direction $e_1$
    \item if $U_t \in [2d-2+e^{\lambda}, Z_{\lambda}]$, then $X$ attempts at time $t$ a jump in direction $-e_1$.
\end{enumerate}
Now let $e$ be the direction in which $X$ is attempting a jump at time $t$.
\begin{enumerate}
    \item If $V_t \in [0,\omega_t(X_{t_-},X_{t_-}+e)]$ then $X_t = X_{t_-}+e$ and the jump succeeds,
    \item if $V_t \in [\omega_t(X_{t_-},X_{t_-}+e), \kappa]$ then $X_t = X_{t_-}$ and the jump fails.
\end{enumerate}

\section{Regeneration structure}
\subsection{Regeneration times VBRW}
The regeneration times we use are adapted from \cite{percolation}. To adapt these regeneration times to VBRW we use the representation in section \ref{alternativeRep}. Note that in contrast to \cite{percolation} the distribution of the regeneration times depend on $\lambda$.\bigskip \\
Let $\mathcal{P}$ be the Poisson process according to which the random walker attempts jumps.
Fix an enumeration $(e_i)$ of the edges in $E_d$ according to an arbitrary rule. Then for each edge $e_i$ we create an infinite number of copies $e_{i,1},e_{i,2},e_{i,3},...$ . We now define a process $(I_t)_{t \geq 0}$, and we call $I_t$ the infected set. Let $I_0 = \emptyset$. Suppose that for some $t \in \mathcal{P} $ the random walk $(X_t)_{t\geq 0}$ attempts a jump along the edge $e_i$, then we add $e_{i,j}$ to $I_t$ for the smallest $j$ such that $e_{i,j} \notin I_{t_-}$.\bigskip \\
Next, for all $t \geq 0$ we assign the lexicographical ordering $\leq_L$ to the edges in $|I_t|$ using the ordering of the edges of $E_d$. Further, let $(N_t)_{t \geq 0}$ be a Poisson process with time dependent intensity $\mu|I_t|$. Whenever a clock of this process rings at time t, we choose an index uniformly at random from $1,...,|I_t|$ and remove the copy of the edge with this index in $|I_t|$ according to the ordering $\leq_L$. Moreover, if the picked copy is of the form $e_{i,1}$ for some $i$ then we refresh the state of the edge $e_i$ in the environment $\omega_t$, i.e. we give $\omega_t(e_i)$ an independent new value according to $q$.\\
For all edges $e_j$ such that $e_{j,1} \notin I_t$, we use independent rate $\mu$ Poisson clocks to determine when the state of the edge in $(\omega_t)_{t \geq 0}$ is refreshed. Note that with this construction $(X_t,\omega_t)_{t \geq 0}$ has indeed the correct transition rates.\\
Let $\tau_0 = 0$, and we define for $n \geq 1$ \[\tau_n := \inf\{t > \tau_{n-1}: \; I_t = \emptyset \text{ and } I_{t'} \neq \emptyset \text{ for some } t' \in (\tau_{n-1},t) \}.\]
Then the times $(\tau_i)_{i \in \N}$ are regeneration times for the process $(X_t)_{t \geq 0}$: this means that $(\tau_{i+1}- \tau_i)_{i \in \N}$ are i.i.d. and $(X_{\tau_{i+1}} - X_{\tau_i})_{i \in \N}$ are i.i.d.
In the following to lighten the notation we will write $\tau$ for $\tau_1$.

\subsection{Properties of the regeneration times}
Our goal is to show that the regeneration times we have defined have exponential moments and that the speed exists.
\begin{lem}\label{deathbirthchain}
    Let $\alpha,\mu >0$.\\
    Let $(A_t)_{t \geq 0 }$ be a continuous time Markov chain on $\mathbb{N}$ with the following $Q$ matrix for $i \geq 0$
    \begin{eqnarray*} 
        q(i,i+1) = \alpha,\\
        q(i+1,i) = (i+1) \mu,\\
        q(i,i) = -\alpha - i \mu.
    \end{eqnarray*}
    Assume $A_0 = 0$ a.s. then let $\tau = \inf\{t \geq 0: A_t = 0 \text{ and } \exists 0 \leq s \leq t: \; A_s \neq 0\}$ then $\tau$ has an exponential tail.\\
    Further, for $T = \inf\{n \geq 1: A_n = 0\}$ where $(A_n)_{n \geq 0}$ is the discretization of $(A_t)_{t \geq 0 }$, we have that $T$ has an exponential tail.
\end{lem}
The proof of this lemma can be found in the appendix.

\begin{lem}[\textit{Exponential tail of the regeneration time}]\label{finiteTau}Let $\mathcal{P}$ be the Poisson process of rate $\kappa \Zl$ according to which the VBRW attempts jumps.\\
    Let $\tau$ be the first regeneration time for a VBRW and $N(\tau) = |\{t\in \mathcal{P}: t \leq \tau \}|$, then 
    \begin{itemize}
        \item $\tau$ has an exponential tail and \(\E[\tau] = \frac{1}{\kappa \Zl} e^{\frac{\kappa \Zl}{\mu}},\)
        \item $N(\tau)$ has an exponential tail and $\E[N(\tau)] = e^{\kappa Z_{\lambda} / \mu} $.
    \end{itemize}
\end{lem}
\begin{proof}
    The existence of an exponential moment for $\tau$ and $N(\tau)$ follows from Lemma \ref{deathbirthchain} with $\alpha = \kappa \Zl$.
    We now want to compute $\E[\tau]$ and $\E[N(\tau)]$\bigskip \\
    First note that $(|I_t|)_{t\geq 0}$ is birth-death process with birth rate $\lambda_i = \kappa \Zl$ and death rate $\mu_i = \mu i$. Using Theorem 7.1 in Chapter 4 of \cite{BirthDeath} we have that 
    \[\E[\tau | |I_0| = 1] = \sum_{i = 1}^{\infty}\frac{(\kappa \Zl)^{i-1}}{\mu^i i!} = \frac{1}{\kappa \Zl} (e^{\frac{\kappa \Zl}{\mu}}-1)\]
    So we get \[\E[\tau | |I_0| = 0] = \frac{1}{\kappa \Zl} e^{\frac{\kappa \Zl}{\mu}}\]
    Next, let $(X_n)_{n\geq 0}$ be the discretization of $|I_t|_{t \geq 0}$. Then $(X_n)_{n\geq 0}$ is a discrete birth-death chain. Let $T_0 = \inf\{n \geq 1: X_n = 0\}$ the first return time to $0$, using the formula for the absorption time of a birth-death chain in Chapter 4 of \cite{BirthDeathChain} we get\\
    \begin{eqnarray*}
        \E[T_0] &=& 1 + \sum_{l = 1}^{\infty} \frac{(\kappa \Zl) ^l }{(\kappa \Zl)(\mu + (\kappa \Zl))\dots((l-1)\mu+(\kappa \Zl))}\frac{(\kappa \Zl)(\mu + (\kappa \Zl))\dots(l\mu+(\kappa \Zl))}{\mu^l l!}\\
        &=&1 + \sum_{l = 1}^{\infty} \left(\frac{\kappa \Zl}{\mu} \right)^l \frac{1}{l!}\frac{(l\mu+(\kappa \Zl))}{\kappa \Zl}\\
        &=&\sum_{l = 0}^{\infty} \left(\frac{\kappa \Zl}{\mu} \right)^l \frac{1}{l!}\frac{(l\mu+(\kappa \Zl))}{\kappa \Zl}\\
        &=&\sum_{l = 1}^{\infty} \left(\frac{\kappa \Zl}{\mu} \right)^{l-1} \frac{1}{(l-1)!} +\sum_{l = 0}^{\infty} \left(\frac{\kappa \Zl}{\mu} \right)^l \frac{1}{l!} \\
        &=& 2 e^{\frac{\kappa \Zl}{\mu}}.
    \end{eqnarray*}
    But we have $T_0 = 2 N(\tau)$,
    so $\E[N(\tau)] = e^{\kappa Z_{\lambda} / \mu} $.
\end{proof}

\begin{lem}[\textit{Existence of the speed}]\label{regenerativSpeed}
    Let $\lambda\geq 0$, $\mu >0$ and let $(X_t)_{t\geq 0}$ be a VBRW on dynamical conductances, then $v(\lambda,\mu) = \lim_{t \rightarrow \infty} \frac{X_t\cdot e_1}{t}$ exists $\P$-a.s. and 
    \[v(\lambda,\mu) = \frac{\E^{\lambda}[X_{\tau}\cdot e_1]}{\E^{\lambda}[\tau]} \quad \P^{\lambda}\text{-a.s.}\]
    Further, 
    \begin{equation*}
        v(\lambda,\mu) = \lim_{t\rightarrow \infty} \frac{\E^{\lambda}[X_t \cdot e_1]}{t}
    \end{equation*}
\end{lem}
\begin{proof} Using Lemma \ref{finiteTau}, we can take over the proof of Proposition 3.1 in \cite{percolation}, if we first fix $\lambda\geq 0$, and then we replace the constant $C$ by a constant $C_{\lambda}$ that depends on $\lambda$.
\end{proof}

\section{VBRW on dynamical conductances has a positive speed}
\subsection{Reversibility}
Our goal is to prove that the random walker $(X_t)_{t \geq 0}$ satisfies an equation that is similar to a detailed balance equation \eqref{eq:detailedBalance} (note that we do not directly talk about reversibility  of $(X_t)_{t \geq 0}$ as it is not a Markov process under $\P$, so we will look at the reversibility of $(\omega_t,X_t)_{t \geq 0}$).\\
In the following $(\omega_t,X_t)_{t \geq 0}$ will denote a VBRW on dynamical conductances $(\omega_t)_{t \geq 0}$. 
We start by define some useful notation. For $\tilde{e}\in E_{d}$ \[\omega_{e \leftarrow v}(\tilde{e}) = v\mathds{1}_{e}(\tilde{e}) + (1-\mathds{1}_{e}(\tilde{e})) \omega(\tilde{e}),\] 
we will write $\mathcal{D}$ for the following set of functions
\[\mathcal{D} = \{f \in C_0(\tilde{\Omega} \times \zd) | \exists E \subset \E_d \text{ with } |E| < \infty \text{ and } \forall \omega \in \confs, \; \forall v \in [0,\kappa], \; \forall z \in \zd, \; f(\omega,z) = f(\omega_{e\leftarrow v})\}.\]
Note that using the Stone-Weierstrass theorem we have that $\mathcal{D}$ is dense in $C_0(\tilde{\Omega} \times \zd)$. 
We can define $\mathcal{L}f$ for $f \in \mathcal{D}$, the generator of $(\omega_t,X_t)_{t \geq 0}$, in the following way
\begin{eqnarray*}
    \mathcal{L}f(\omega,x) = \sum_{y \sim x} \omega_t^{\lambda}([x],[y]) \left( f(\omega,y) - f(\omega,x)\right) + \mu \sum_{e \in E_{d}} \left[ \int_{0}^{\kappa} f(\omega_{e \leftarrow v}, x)dq(v) - f(\omega, x) \right].
\end{eqnarray*}
Let $\pi$ be the following measure on $\zd$:
\[\pi(x) = e^{\lambda ( 2 x \cdot e_1)} \quad \forall x \in \zd.\]

\begin{lem}
    $(\omega_t, X_t)_{t \geq 0}$ is a reversible Markov process. 
    Its reversible measure is given by 
    \[
        \rho = \Q \otimes \pi,
    \]
    with $Q$ as define before the product measure of $q$ on the edges.
\end{lem}
\begin{proof}
    Let $f,g \in \mathcal{D}$, with 
    \[\int |f| d\rho< \infty \quad \text{ and } \quad \int |g| d\rho < \infty.\]
    We want to show that $\langle \L f,g\rangle_{\rho}  = \langle f, \L g \rangle_{\rho}$\\
    Let $E_f = \{e \in E_d| \exists \omega \in \confs, \; \exists z \in \zd,\; \exists v \in [0,\kappa] \text{ such that } f(\omega,z) \neq f(\omega_{e \leftarrow v},z)\}$, and analogously define $E_g$.\\
    Using that $f,g \in \mathcal{D} $ we have that $||f||_{\infty} < \infty $, $||g||_{\infty} < \infty $, and $|E_f \cup E_g| < \infty$.\\
    We have that $|\L f (\omega,x)| \leq 2 C_f (2d\kappa + \mu |E_{f}| )$, so we have that
    \begin{eqnarray*}
        \int |(\L f)g| d\rho \leq  2 C_f (2d\kappa + \mu |E_{f}| )\int |g| d\rho < \infty.
    \end{eqnarray*}
    So we can apply Fubini's theorem

    \begin{eqnarray}
        \langle \L f,g\rangle_{\rho} &=& \int_{\confs \times \zd} (\L f)(\omega,x) g(\omega,x) d\rho(\omega,x) \nonumber\\
        &=&\sum_{x\in \zd} \int_{\confs} (\L f)(\omega,x) g(\omega,x) e^{2\lambda x \cdot e_1} dQ(\omega) \quad \text{ (Fubini)} \nonumber\\
        &=& \sum_{x\in \zd} \int_{\confs}\sum_{e \in E_{f} \cup E_g} \mu \left[ \int_{0}^{\kappa} f(\omega_{e \leftarrow v},x) - f(\omega,x) dq(v)\right]g(\omega,x) e^{2\lambda x \cdot e_1} dQ(\omega)\nonumber\\ 
        && +\sum_{x\in \zd} \int_{\confs} \sum_{y \sim x} \omega^{\lambda}(x, y) \left[f(\omega,y) - f(\omega,x) \right] g(\omega,x) e^{2\lambda x \cdot e_1} dQ(\omega) \quad \text{(linearity)}\nonumber\\ 
        &=& \mu \sum_{x\in \zd} \int_{\confs}\sum_{e \in E_f \cup E_g} e^{2\lambda x \cdot e_1} \int_{0}^{\kappa} f(\omega_{e \leftarrow v},x) g(\omega,x) dq(v) dQ(\omega) \label{firstterm}\\
        &-& \mu |E_{f} \cup E_g| \sum_{x\in \zd} \int_{\confs} e^{2\lambda x \cdot e_1} f(\omega,x) g(\omega,x) dQ_(\omega) \nonumber\\
        &+&\int_{\omega \in \confs} \sum_{x\in \zd} \sum_{y \sim x} e^{\lambda(x+y)\cdot e_1}\omega(x,y) \\
        && \quad \quad \left[f(\omega,y)g(\omega,x)- f(\omega,x) g(\omega,x)\right] dQ(\omega)\, \text{(Fubini)}. \label{lastline}
    \end{eqnarray}
    Next we are first interested in the term in line \eqref{firstterm} , where we apply again Fubini's theorem
    \begin{eqnarray*}
        && \sum_{x\in \zd} \int_{\confs}\sum_{e \in E_{f} \cup E_g} e^{2\lambda x \cdot e_1} \int_{0}^{\kappa} f(\omega_{e \leftarrow v},x) g(\omega,x) dq(v) dQ(\omega) \\
        &&= \sum_{x\in \zd} e^{2\lambda x \cdot e_1}\sum_{e \in E_{f} \cup E_g}  \int_{0}^{\kappa} \int_{ \confs} f(\omega_{e \leftarrow v},x) g(\omega,x) dQ(\omega) dq(v)\\
        &&=\sum_{x\in \zd}e^{2\lambda x \cdot e_1}\sum_{e \in E_f \cup E_g} \int_{0}^{\kappa} \int_{0}^{\kappa} \int_{[0,\kappa)^{E_{d} \setminus \{e\}}} f(\omega_{e \leftarrow v},x) g(\omega_{e \leftarrow w},x) dq^{E_{d} \setminus \{e\}} dq(w) dq(v)\\
        &&=\sum_{x\in \zd}e^{2\lambda x \cdot e_1}\sum_{e \in E_f \cup E_g} \int_{0}^{\kappa} \int_{\omega \in \confs} f(\omega,x) g(\omega_{e \leftarrow w},x) dQ(\omega) dq(w).
    \end{eqnarray*}
    Fubini's theorem can be used as 
    \begin{eqnarray*}
        \sum_{x\in \zd} \int_{\confs}\sum_{e \in E_{f} \cup E_g} e^{2\lambda x \cdot e_1 } \int_{0}^{\kappa} |f(\omega_{e \leftarrow v},x) g(\omega,x)| dq(v) dQ(\omega) \leq |E_f \cup E_g| C_f \int e^{2\lambda x} |g| d\rho < \infty.
    \end{eqnarray*}
    The second term in the line \eqref{lastline} is symmetric in $f$ and $g$. So we only need to look at the following:
    \begin{eqnarray*}
        &&\int_{\confs} \sum_{x\in \zd} \sum_{y \sim x} e^{\lambda(x+y) \cdot e_1}\omega(x,y) f(\omega,y)g(\omega,x)dQ(\omega)\\
        &&= \int_{\confs} \sum_{x, y\in \zd} \sum_{y \sim x} e^{\lambda(x+y)\cdot e_1}\omega(x,y) \mathds{1}_{x\sim y} \frac{\left(f(\omega,y)g(\omega,x) + f(\omega,x)g(\omega,y)\right)}{2} \;dQ(\omega)\\
        &&= \int_{\confs} \sum_{x\in \zd} \sum_{y \sim x} e^{\lambda(x+y)\cdot e_1}\omega(x,y) f(\omega,x)g(\omega,y)dQ(\omega).
    \end{eqnarray*}
    This way we get that 
    \[
        \langle \L f,g\rangle_{\rho}  = \langle f, \L g \rangle_{\rho}
    \]
    and $\rho$ is a reversible measure for the process $(\omega_t, X_t)_{t \geq 0}$.
\end{proof}

\begin{lem} \label{reversibility_X}
    Let $x,y \in \zd$ and $t \geq 0$ then 
    \begin{equation}\label{eq:detailedBalance}
        \pi(x) \P(X_t = y | X_0 = x) = \pi(y) \P(X_t = x | X_0 = y).
    \end{equation}
\end{lem}
\begin{proof}
    Let $\P_{\rho}(.) = \int \P_x(.) d\rho$ and $t >0$ then by reversibility under $\P_{\rho}$
    \begin{equation*}
        \left((\omega_t,X_t),(\omega_0,X_0) \right) \overset{d}{=} \left((\omega_0,X_0),(\omega_t,X_t)\right).
    \end{equation*}
    Let $A= \{((a,b),(c,d)) \in (\confs \times \zd)^2: a=x \text{ and } c = y\}$ then 
    \begin{equation*}
        \P_{\rho}\left( \left((\omega_t,X_t),(\omega_0,X_0) \right)\in A\right) = \P_{\rho}\left( \left((\omega_0,X_0),(\omega_t,X_t) \right)\in A\right),
    \end{equation*}
    but 
    \begin{equation*}
        \P_{\rho}\left( \left((\omega_t,X_t),(\omega_0,X_0) \right)\in A\right) = \P_{\rho}\left(X_t = x, X_0 =y\right) =e^{2\lambda y \cdot e_1}\P\left(X_t = x| X_0 =y\right)
    \end{equation*}
    and 
    \begin{equation*}
        \P_{\rho}\left( \left((\omega_0,X_0),(\omega_t,X_t) \right)\in A\right) = \P_{\rho}\left(X_t = y, X_0 =x\right) = e^{2\lambda x \cdot e_1}\P\left(X_t = y| X_0 =x\right).\\
    \end{equation*}
    So 
    \begin{equation*}
        \pi(x) \P(X_t = y | X_0 = x) = \P(X_t = x | X_0 = y).
    \end{equation*}
\end{proof}

\begin{cor} \label{cor:rev}
    Let $x,y \in \zd $ then 
    \begin{equation}\label{eq:correversibility}
        \P(X_t = y + x | X_0 = y)= e^{2 \lambda ( x \cdot e_1)} \P(X_t = y - x | X_0 = y) \quad \forall t \geq 0 .
    \end{equation}
\end{cor}
\begin{proof}
    First note that by Lemma \ref{reversibility_X} we have \[
    \pi(y) \P(X_t = y + x | X_0 = y) = \pi(y+x) \P(X_t = y  | X_0 = y +x)\]
    and $\pi(y) = e^{2\lambda (y \cdot e_1)} $, $\pi(x+y) = e^{2\lambda ((x+y) \cdot e_1)}$ so,
    \[ \P(X_t = y + x | X_0 = y) = e^{2\lambda (x \cdot e_1)} \P(X_t = y | X_0 = y + x).\]
    Next, we see that as the measure $\Q$ is invariant and by the definition of $\P$ we have that under $\P$ the law of $(\omega_t)_{t \geq 0}$ is translation invariant for all $t \geq 0$. So we have that the distribution of $(X_t)_{t \geq 0}$ under $\P$ is translation invariant, so  \[\P(X_t = y  | X_0 = y +x) = \P(X_t = y - x  | X_0 = y)\]
    and we see that \eqref{eq:correversibility} is true.
\end{proof}

\subsection{CLT}
\begin{thm}[\textit{CLT}] \label{CLTVBRW}
    \begin{eqnarray}
        \frac{X_t \cdot e_1 - v(\lambda,\mu) t}{\sqrt{t}}  \xrightarrow[t \to \infty]{(d)} \mathcal{N}\left(0,\frac{Var(X_{\tau_1} \cdot e_1)}{\E[\tau_1]}\right)
    \end{eqnarray}
\end{thm}

\begin{proof}

\ifdefined\diss
    First note that as $(\tau_n)_{n \geq 0}$ are regeneration times and as $N$ has some exponential tails, we have that $X_{\tau_{n}}$ is a sum of i.i.d. random variables.
    \begin{eqnarray} \label{iidCLT1}
        \frac{X_{\tau_n} \cdot e_1 - v(\lambda,\mu) n \E[\tau_1]}{\sqrt{n Var(X_{\tau_1} \cdot e_1)}}  \xrightarrow[n \to \infty]{(d)} \mathcal{N}(0,1)
    \end{eqnarray}
    For $t \geq 0$, let $l(t)= \max\{l \in \N: \tau_l \leq t\}$. $l(t)$ is a renewal process, so by renewal argument we get that
    \begin{eqnarray} \label{eq:renewalVBRW}
        \lim_{t \rightarrow \infty} l(t)/t = \E[\tau]^{-1} \quad \text{ a.s}.
    \end{eqnarray} 
    Next we decompose:
    \begin{eqnarray}
        \frac{X_t \cdot e_1 - v(\lambda,\mu) t}{\sqrt{t}} = \frac{X_t \cdot e_1 - X_{\tau_{l(t)}} \cdot e_1 }{\sqrt{t}} + \frac{X_{\tau_{l(t)}} \cdot e_1 - v(\lambda,\mu) \E[\tau_{l(t)}]}{\sqrt{t}} + \frac{v(\lambda,\mu) (\E[\tau_{l(t)}] -t)}{\sqrt{t}}
    \end{eqnarray}
    Now note that $|X_{t} \cdot e_1 - X_{\tau_{l(t)}} \cdot e_1| \leq N([\tau_{l(t)},\tau_{l(t)+1}])$, but $N([\tau_{l(t)},\tau_{l(t)+1}])$ has the same law then $N$ has some exponential moments so for all $\varepsilon >0 $
    \begin{eqnarray}
        \P(|X_{t} \cdot e_1 - X_{\tau_{l(t)}} \cdot e_1| \geq \sqrt{t} \varepsilon) \leq \frac{\E[N^2]}{t \varepsilon^2} \xrightarrow[t \to \infty]{} 0,
    \end{eqnarray}
    so
    \begin{eqnarray}
        \frac{X_t \cdot e_1 - X_{\tau_{l(t)}} \cdot e_1 }{\sqrt{t}} \xrightarrow[t \to \infty]{(d)} 0.
    \end{eqnarray}
    For the second term we have that 
    \begin{eqnarray*}
        0\leq t - \E[\tau_{l(t)}] \leq \E[\tau_{l(t) + 1}] - \E[\tau_{l(t)}] = \E[\tau_{1}],
    \end{eqnarray*}
    so 
    \begin{eqnarray}
        \frac{v(\lambda,\mu) (\E[\tau_{l(t)}] -t)}{\sqrt{t}} \xrightarrow[t \to \infty]{(d)} 0.
    \end{eqnarray}
    Now we decompose
    \begin{eqnarray}
        \frac{X_{\tau_{l(t)}} \cdot e_1 - v(\lambda,\mu) \E[\tau_{l(t)}]}{\sqrt{t}} = \frac{X_{\tau_{l(t)}} \cdot e_1 - v(\lambda,\mu) \E[\tau_{l(t)}]}{\sqrt{l(t)}} \sqrt{\frac{l(t)}{t}}.
    \end{eqnarray}
    Using that by equation \eqref{eq:renewalVBRW} that $\sqrt{\frac{l(t)}{t}} \xrightarrow[t \to \infty]{} \frac{1}{\sqrt{\E[\tau_1]}}$ a.s. 
    Using \eqref{iidCLT1}
    \begin{eqnarray}
        \frac{X_{\tau_{l(t)}} \cdot e_1 - v(\lambda,\mu) \E[\tau_{l(t)}]}{\sqrt{l(t)}} = \frac{X_{\tau_{l(t)}} \cdot e_1 - v(\lambda,\mu) l(t) \E[\tau_1 ]}{\sqrt{l(t)}}\xrightarrow[t \to \infty]{(d)} \mathcal{N}\left(0,Var(X_{\tau_1})\right) .
    \end{eqnarray}
    Altogether we get 
    \begin{eqnarray}
        \frac{X_t \cdot e_1 - v(\lambda,\mu) t}{\sqrt{t}}  \xrightarrow[t \to \infty]{(d)} \mathcal{N}\left(0,\frac{Var(X_{\tau_1} \cdot e_1)}{\E[\tau_1]}\right).
    \end{eqnarray}

\else

    The proof goes analogously to the proof of Theorem 3.1 in \cite{unbiasedPercolation}

\fi
\end{proof}
\subsection{Positivity of the speed}
    \begin{lem}[\textit{Positivity of the expectation at a finite time}] We have
        $\E[X_t \cdot e_1] > 0$ for all $t \geq 0$.
    \end{lem}
    \begin{proof}
        \begin{eqnarray*}
            \E[X_t \cdot e_1] &=& \sum_{n = 0}^{\infty} \sum_{y \cdot e_1 = n} n \left(\P(X_t \cdot e_1 = y) - \P(X_t \cdot e_1 = -y)\right)\\
            &\geq& \sum_{n = 0}^{\infty} \sum_{y \cdot e_1 = n} n \P(X_t \cdot e_1 = -y)\left(e^{2\lambda n} - 1\right),\\
        \end{eqnarray*}
        where we used Corollary \ref{cor:rev}.\\
        Then using that for $t>0$, $\P(X_t \cdot e_1 = 0) < 1 $, we have that 
        $\E[X_t \cdot e_1] > 0$.
    \end{proof}

    \begin{lem}[\textit{Non-negative speed}] \label{speednonNegative}
        \[v(\lambda, \mu ) \geq 0.\]
    \end{lem}
    \begin{proof}
        $v(\lambda, \mu )  = \lim_{t \rightarrow + \infty} \frac{\E[X_t\cdot e_1]}{t}$. But $\E[X_t \cdot e_1] > 0$ for $t > 0$.
    \end{proof}

    \begin{thm}[\textit{Positive speed}]\label{positiveSpeed} Let $\lambda >0$, then
        \[v(\lambda, \mu ) > 0.\]
    \end{thm}
    \begin{proof}
        Let $Y_t = -X_t$ for all $t \geq 0$.\\
        Note that the increments $((X_{\tau_{n+1}} - X_{\tau_n})\cdot e_1)_{n \geq 0}$ are i.i.d. such that $(X_{\tau_n} \cdot e_1)_{n \geq 0}$ is a one-dimensional random walk. Let us assume that $\E[X_{\tau}] = 0$ (so $v(\lambda, \mu ) = 0$), then using the CLT (Theorem \ref{CLTVBRW})
        \begin{equation*}
            \frac{X_{t} \cdot e_1}{\sqrt{t}} \rightarrow_{t \rightarrow +\infty} \mathcal{N}(0,\sigma^2) \text{ in distribution.}
        \end{equation*}
        Note that $\sigma^2 = \E[X_{\tau_1}^2]/\E[\tau] > 0$.
        Similarly, 
        \begin{equation*}
            \frac{Y_{t} \cdot e_1}{\sqrt{t}} \rightarrow_{t \rightarrow +\infty} \mathcal{N}(0,\sigma^2) \text{ in distribution.}
        \end{equation*} 
        We now recall that Lemma \ref{cor:rev} gives us for $ y \in \zd$
        \begin{eqnarray*}
            \P(X_t = y) \geq e^{2\lambda y \cdot e_1} \P(X_t = -y) = e^{2\lambda y \cdot e_1} \P(Y_t = y).
        \end{eqnarray*}
        Then for $\alpha >0$ we have that
        \begin{eqnarray*}
            \P(Y_t \cdot e_1 \leq -\alpha \sqrt{t}) \geq e^{2\lambda \alpha \sqrt{t}} \P(X_t \cdot e_1 \leq -\alpha \sqrt{t}).
        \end{eqnarray*}
        So 
        \begin{eqnarray}\label{contradiction}
            \P(\frac{Y_t \cdot e_1}{\sqrt{t}} \leq -\alpha) \geq e^{2\lambda \alpha \sqrt{t}} \P(\frac{X_t \cdot e_1}{\sqrt{t}}\leq -\alpha),
        \end{eqnarray}
        but by the convergence in distribution we have that 
        \begin{eqnarray}\label{setconv}
            \lim_{t \rightarrow + \infty} \P(\frac{Y_t \cdot e_1}{\sqrt{t}} \leq -\alpha) = \lim_{t \rightarrow + \infty} \P(\frac{X_t \cdot e_1}{\sqrt{t}} \leq -\alpha) = \Phi(-\alpha)
        \end{eqnarray}
        with $\Phi(-\alpha)$ the probability that a $\mathcal{N}(0,\sigma^2)$ distributed random variable is smaller than $-\alpha$.
        For \eqref{contradiction} and \eqref{setconv} to be satisfied we would need $\Phi(-\alpha) = 0$ which is not possible as $\sigma^2 > 0$.\\
        So it follows by contradiction that $\E[X_{\tau}] \neq 0$, and with Lemma \ref{speednonNegative} we get the claim.
    \end{proof}

\section{The asymptotics of the speed of the (N-)VBRW}
In this section we want to understand the behavior of the speed for large biases. We recall that we are assuming that the conductances are upper-bounded by $\kappa$, but they do not need to be bounded away from $0$. The couplings presented in this section are adapted from the couplings in \cite{percolation}.
\subsection{Time-normalizing the variable speed process (NVBRW)}
In this section we want to introduce the normalized variable speed biased random walk on dynamical conductances, and explain why studying this process makes sense.
\begin{defi}\label{defNVBRW}
    Let $\lambda,\mu >0$.\\
    Let $(\tilde{\eta}_t,\widetilde{X}_t)_{t\geq 0}$ be a biased variable speed random walk on dynamical random conductances with parameter $\lambda$ and $Z_{\lambda}\mu$, with $\lambda,\mu >0 $ and $Z_{\lambda} = e^{\lambda}+e^{-\lambda}+2d-2$ as before.
    Then we call \[(\eta_t,X_t) = (\tilde{\eta}_{tZ_{\lambda}^{-1}},\widetilde{X}_{tZ_{\lambda}^{-1}})\]
    the normalized variable speed biased random walk (NVBRW) on dynamical conductances with parameter $\lambda$ and $\mu$.\bigskip\\
    We will denote the speed of the NVBRW on dynamical conductances with parameter $\lambda$ and $\mu$
    \[\hat{v}(\lambda,\mu) := \lim_{t\rightarrow\infty}\frac{X_t}{t}.\]
\end{defi}

The first advantage of this process is that for $\lambda \rightarrow \infty$ the jump rate of the walker is not diverging. This will allow us to study the asymptotic behavior of the speed of a NVBRW on dynamical conductances $\hat{v}(\lambda,\mu)$ and then link those result to the speed of a VBRW on dynamical conductances $v(\lambda,Z_{\lambda} \mu)$ using the following lemma.
\begin{lem}\label{speedNVvsV}
    For all $\lambda,\mu >0$:
    \[v(\lambda,Z_{\lambda} \mu) = Z_{\lambda} \hat{v}(\lambda,\mu)\]
    and in particular $\hat{v}(\lambda,\mu)$ exists $\P$-a.s.
\end{lem}
\begin{proof}
    Let $(\widetilde{X}_t,\tilde{\eta}_t)_{t\geq 0}$ and $(X_t,\eta_t)_{t\geq 0}$ be as in Definition \ref{defNVBRW}, then 
    \begin{eqnarray*}
        v(\lambda,Z_{\lambda}\mu) &=&  \lim_{t \rightarrow \infty}\frac{\widetilde{X}_t}{t}
        = \lim_{t \rightarrow \infty}\frac{\widetilde{X}_{Z_{\lambda}^{-1}t}}{Z_{\lambda}^{-1}t}\\
        &=& Z_{\lambda}\lim_{t \rightarrow \infty}\frac{X_t}{t} 
        = Z_{\lambda} \hat{v}(\lambda,\mu).
    \end{eqnarray*}
\end{proof}
But there is another advantage of this process, when one wants to compare two VBRW on dynamical conductances with different biases. Let $X$ be a VBRW with parameters $\lambda+\varepsilon,\mu$ and $Y$ be a VBRW with parameters $\lambda,\mu$, then $X$ is attempting more jumps and so discovers the environment faster. This means that from the perspective of $X$ the environment is refreshing slower than from the perspective of $Y$ so to compare both processes we would like to speed up the environment of $X$ and time-change the processes so that both walkers attempt jumps at the same rate.
\begin{rmk}\label{remarkNormalization}
    Note that for a fixed $\lambda$ the NVBRW process is just a constant time change of the VBRW process. This means that we recover some properties:
    \begin{itemize}
        \item[(i)] For all $\lambda,\mu>0$ we have that $\hat{v}(\lambda,\mu)>0$.
        \item[(ii)] Just as for the VBRW on dynamical conductances we can define an infected set $(I_t)_{t\geq 0}$ and use this set to define the same way regeneration times $(\tau_n)_{n\geq 0}$. Note that the rate at which a jump is attempted is $\kappa$ so the distribution of $(|I_t|)_{t\geq0}$ is not depending on $\lambda$. In particular one gets that $\E[N(\tau_1)] = e^{\kappa/\mu}$ and $\E[\tau_1] = \frac{1}{\kappa} e^{\kappa/\mu}$.
        \item[(iii)] We can do for the NVBRW the same alternative construction as for the VBRW, we only need to choose the Poisson process $\mathcal{P}$ tho be with rate $\kappa$ instead of $\Zl \kappa$. 
    \end{itemize}
\end{rmk}
\begin{rmk}
    Note that since $0$ is not excluded in the environment, the NVBRW on dynamical percolation is a special case of the studied model (the alternative construction of the NVBRW corresponds to the construction of the random walk in \cite{percolation}). For this case we recover the same asymptotics that were obtained in \cite{percolation}.
\end{rmk}
\subsection{The one-dimensional totally asymmetric NVBRW ($\lambda = +\infty$)}
Let $(\omega_t)_{t \geq 0}$ be dynamical conductances on $\mathbb{Z}$ with parameter $q$ and $\mu >0$. We define the totally asymmetric NVBRW $(A_t)_{t \geq 0}$ on $(\omega_t)_{t \geq 0}$ as follows. Let $\mathcal{P}$ be a Poisson process with rate $\kappa$ then for every point $t \in \mathcal{P}$, $A_t$ attempts a jump in direction $e_1$. This means that for every $t \in \mathcal{P}$ we sample an independent uniformly distributed random variable $U_t$ on $[0,\kappa]$. If $U_t \leq \omega_t(A_{t^-}, A_{t^-}+1)$ then $A_{t} = A_{t^-} +1$ otherwise $A_{t} = A_{t^-}$. We assume  $A_0 = 0$ a.s.\bigskip \\
For some time evolution of the conductances $\omega = (\omega_t)_{t \geq 0}$ we will use as for the previous model $P_{\omega}$ and $E_\omega$ to denote the probability measure and the expectation corresponding to the random walk on the environment $\omega$. $\P$ and $\E$ will be used for the joint process.

\begin{thm} 
    Let $(A_t)_{t \geq 0}$ be a totally asymmetric NVBRW on dynamical conductances with parameter $q$ and $\mu >0$. If $\tilde{\omega}$ is a random variable distributed according to $q$ then,
    \begin{equation}\label{totallyAsymetricSpeed}
        v_A(\mu) := \lim_{t\rightarrow \infty} \frac{A_t}{t} = \frac{\E\left[\frac{\tilde{\omega}}{\mu + \tilde{\omega}}\right]}{\E\left[\frac{1}{\mu + \tilde{\omega}}\right]} \quad \mathbb{P}- \text{a.s.}
    \end{equation}
\end{thm}
\begin{proof}
    Let $T = \inf\{t \geq 0 : A_t \neq 0\}$\\
    As the process $(A_s)_{s \geq t}$ is independent of $(\omega(y))_{t\leq s}$ for all $y \leq A_t$, we have that $v_A = \frac{1}{\E[T]}$. \bigskip \\
    Let $\omega = (\omega_t)_{t\geq 0}$ be a time evolution of the conductances.
    \begin{eqnarray*}
        E_{\omega}[T] &=& \int_{0}^{\infty} P_{\omega}(T \geq x) dx = \int_{0}^{\infty} e^{-\int_{0}^{x} \omega_u(0,1) du} dx.\\
    \end{eqnarray*}
    We now compute $\E \left[ e^{-\int_{0}^{x} \omega_u(0,1) du}\right]$.\\ 
    Let $\mathcal{S}$ be the Poisson process that gives the times at which the environment $\omega$ refreshes its values on the edge $(0,1)$. Let $(T_n)_{n\geq 0}$ be the points of $\mathcal{S}$ (with $T_0=0$).\\
    Then for $\tilde{\omega}$ an independent random variable distributed according to $q$.
    \begin{eqnarray*}
        \E[T] &=& \E\left[\int_{0}^{\infty}e^{-\int_{0}^{x} \omega_u(0,1) du} dx\right] = \E\left[\sum_{n=0}^{\infty}\int_{T_n}^{T_{n+1}}e^{-\int_{0}^{x} \omega_u(0,1) du} dx\right]\\
        &=& \sum_{n=0}^{\infty}\E\left[e^{-\sum_{k=1}^{n}(T_k-T_{k-1})\omega_{T_{k-1}}(0,1) } \int_{0}^{T_{n+1}-T_{n}} e^{-x \tilde{\omega}} dx\right]\\
        &=& \sum_{n=0}^{\infty} \E\left[e^{-\sum_{k=1}^{n}(T_k-T_{k-1})\omega_{T_{k-1}}(0,1) } \left( \mathds{1}_{\tilde{\omega} \neq 0} \frac{1-e^{-(T_{n+1}-T_{n}) \tilde{\omega}}}{\tilde{\omega}} + \mathds{1}_{\tilde{\omega} = 0} (T_{n+1}-T_n)\right)\right]\\
        &=& \sum_{n=0}^{\infty}\prod_{k=1}^{n}\E\left[e^{-(T_k-T_{k-1})\omega_{T_{k-1}}(0,1) }\right] \E\left[\mathds{1}_{\tilde{\omega} \neq 0} \frac{1-e^{-(T_{n+1}-T_{n}) \tilde{\omega}}}{\tilde{\omega}} + \mathds{1}_{\tilde{\omega} = 0} (T_{n+1}-T_n)\right]\\
        &=&\sum_{n=0}^{\infty}\left( \E\left[\frac{\mu}{\mu + \tilde{\omega}}\right] \right)^n \E\left[\mathds{1}_{\tilde{\omega} \neq 0} \frac{1-e^{-(T_{n+1}-T_{n}) \tilde{\omega}}}{\tilde{\omega}} + \mathds{1}_{\tilde{\omega} = 0} (T_{n+1}-T_n)\right] \quad \quad \text{Using \eqref{firstTerm}}\\
        &=&\E\left[\mathds{1}_{\tilde{\omega} \neq 0} \frac{1-e^{-(T_{n+1}-T_{n}) \tilde{\omega}}}{\tilde{\omega}} + \mathds{1}_{\tilde{\omega} = 0} (T_{n+1}-T_n)\right] \frac{1}{1-\E\left[\frac{\mu}{\mu + \tilde{\omega}}\right]}.\\
    \end{eqnarray*}
    Note that we can pull the infinite sum out of the expectation as the terms are all non negative. Further we can split the expectation of the product as the product of the expectations as the different values an edge takes are independent.
    We now compute 
    \begin{eqnarray*}
        \E\left[\mathds{1}_{\tilde{\omega} \neq 0}\frac{1-e^{-(T_{n+1}-T_{n}) \tilde{\omega}}}{\tilde{\omega}}\right] &=& \E\left[ \E[\mathds{1}_{\tilde{\omega} \neq 0}\frac{1-e^{-(T_{n+1}-T_{n}) \tilde{\omega}}}{\tilde{\omega}} | \tilde{\omega}]\right]\\
        &=&\E\left[\mathds{1}_{\tilde{\omega} \neq 0} \tilde{\omega}^{-1} - \frac{1}{\tilde{\omega}} \int_{0}^{\infty} \mu e^{-t (\tilde{\omega} + \mu)}dt \right]\\
        &=&\E\left[\mathds{1}_{\tilde{\omega} \neq 0} \tilde{\omega}^{-1} - \frac{1}{\tilde{\omega}} \frac{-\mu}{\mu + \tilde{\omega}}\right]
        =\E\left[\mathds{1}_{\tilde{\omega} \neq 0}\frac{1}{\mu + \tilde{\omega}}\right].\\
    \end{eqnarray*}
    This gives us 
    \begin{equation*}
        \E\left[\mathds{1}_{\tilde{\omega} \neq 0} \frac{1-e^{-(T_{n+1}-T_{n}) \tilde{\omega}}}{\tilde{\omega}} + \mathds{1}_{\tilde{\omega} = 0} (T_{n+1}-T_n)\right] = \E\left[\frac{1}{\mu + \tilde{\omega}}\right].\\
    \end{equation*}
    Further 
    \begin{eqnarray}\label{firstTerm}
        \E\left[e^{-(T_k-T_{k-1})\omega_{T_{k-1}}(0,1) }\right] &=& \P(\tilde{\omega} = 0) + \E\left[\E\left[\mathds{1}_{\tilde{\omega} \neq 0}e^{-(T_k-T_{k-1})\tilde{\omega} }|\tilde{\omega}\right]\right]\nonumber\\
        &=&\P(\tilde{\omega} = 0) + \E\left[\mathds{1}_{\tilde{\omega} \neq 0} \int_{0}^{\infty} \mu e^{-t \tilde{\omega}} e^{-\mu t} dt\right]\nonumber\\
        &=& \P(\tilde{\omega} = 0)+\E\left[\mathds{1}_{\tilde{\omega} \neq 0} \frac{\mu}{\mu+\tilde{\omega}}\right]= \E\left[\frac{\mu}{\mu+\tilde{\omega}}\right] \nonumber.\\
    \end{eqnarray}
    Altogether we get that 
    \begin{eqnarray*}
        \E[T] &=&\E\left[\frac{1}{\mu + \tilde{\omega}}\right]\frac{1}{1-\E\left[\frac{\mu}{\mu + \tilde{\omega}}\right]}
        =\frac{\E\left[\frac{1}{\mu + \tilde{\omega}}\right]}{\E\left[\frac{\tilde{\omega}}{\mu + \tilde{\omega}}\right]}.
    \end{eqnarray*}
\end{proof}
\subsection{The asymptotic speed of the one-dimensional NVBRW}
The goal is now to show that in one-dimension the speed of the NVBRW  $\hat{v}(\lambda,\mu)$ is converging exponentially fast in the bias $\lambda $ to the speed of the totally asymmetric process $v_A(\mu)$, for $\lambda \rightarrow +\infty$. \\
\begin{rmk}
    It is possible to define for the process totally asymmetric process $(\omega_t,A_t)_{t\geq 0}$ an infected set $(I^A_t)_{t\geq 0}$ and define then the regeneration times $\tau_{n+1} = \inf\{t \geq \tau_n: I^A_t = \emptyset \text{ and } \exists \tau_n \leq s \leq t: \; I^A_t \neq \emptyset\}$, we then have the same way as before $v_A = \frac{\E[A_{\tau_1}]}{\E[\tau_1]}$.
\end{rmk}
\begin{lem} \label{lem14}  Let $d=1$. 
    Let $d =1$, $\lambda >0$ and $\mu> 0$, then for $\tilde{\omega} \sim q$ and a constant $C_{\mu}$ that only depends on $\mu$ we have 
     \[0 \leq \frac{\E\left[\frac{\tilde{\omega}}{\mu + \tilde{\omega}}\right]}{\E\left[\frac{1}{\mu + \tilde{\omega}}\right]} - \hat{v}(\lambda, \mu) \leq C_{\mu} e^{-2\lambda}.\]
\end{lem}
\begin{proof}
    Let $(\omega_t,Y_t)_{t \geq 0}$ be a NVBRW on dynamic random environment with parameter $\lambda >0$ and $\mu>0$.\\
    Further, let $(A_t)_{t \geq 0}$ be a totally asymmetric random walk on $(w_t)_{t \geq 0}$ a dynamical environment. We then couple $(\omega_t,Y_t)_{t\geq 0}$ and $(w_t,A_t)_{t\geq 0}$ in the following way. Let $w_0 = \omega_0$.\\
    Let $\mathcal{P}$ be a Poisson process of rate $\kappa$ then at the points of $\mathcal{P}$ both process will attempt a jump. 
    Take $t$ to be a point of $\mathcal{P}$ and sample $U_t \sim Unif[0, e^{\lambda} + e^{-\lambda}]$, then 
    \begin{enumerate}
        \item If $U_t \in [0, e^{\lambda})$ then $Y$ attempts a jump in direction $e_1$ otherwise in direction $-e_1$
        \item $A$ always attempts a jump in direction $e_1$.
    \end{enumerate}
    Let $V_t \sim Unif[0, \kappa]$ and assume $X$ attempts a jump in direction $e$ 
    \begin{enumerate}
        \item If $V_t \in [0, \omega_t({Y_{t^-}},Y_{t^-}+e)]$ then the process $Y$ jumps in direction $e$ in $t$.
        \item If $V_t \in [0, w_t({A_{t^-}},A_{t^-}+e_1)]$ then the process $A$ jumps to the right in $t$.
    \end{enumerate}
    Let $(I^A_t)_{t\geq 0}$ be the infected set of $A$ and $(I^Y_t)_{t\geq 0}$ be the infected set of $Y$. Then up to $S = \inf\{t \geq 0: A_t \neq Y_t\}$ we can decide to remove the edges in $I^A$ and $I^Y$ the same way, in order to have that $I^A_t = I^Y_t$. This way we can refresh the edge the edges in $\omega_t$ and $w_t$ at the same time points, and we decide to refresh them such that $\omega_t = w_t$ for all $t \in [0,S]$. At time $S$ we stop the coupling, then the $\omega_t$ and $w_t$ evolve independently, and $A$ and $Y$  continue jumping at the points of $\mathcal{P}$ but decide of the direction of the jump and the success using independent random variables for $t \geq S$. Let $\tau^A_1$ be the first regeneration time of $A$ and $\tau^Y_1$ be the first regeneration time of $Y$.\\
    We also note that in our construction $\tau^A_1$ and $\tau^Y_1$ do not depend on $\lambda$ anymore, as there transition rates are given by $q(i,i+1) = \kappa$ and $q(i+1,i) = \mu(i+1)$ for $i \geq 0$.\bigskip \\
    Now let $N(t_0) = |\{t \in \mathcal{P}: t \leq t_0\}|$ for $t_0 \geq 0$. \\
    Then 
    \begin{eqnarray*}
        \P(\tau^Y_1 \geq S) &\leq& \sum_{n=0}^{\infty}\P(N(\tau^Y_1)=n) (n \frac{ e^{-\lambda}}{e^{\lambda}+e^{-\lambda}})= \frac{ e^{-\lambda}}{e^{\lambda}+e^{-\lambda}} \E[N(\tau^Y_1)].
    \end{eqnarray*}
    Now recall that $\E[N(\tau_1)] < \infty$ by Remark \ref{remarkNormalization} so 
    \[\E[A_{\tau^A_1}-  Y_{\tau^Y_1}] \leq 0 \cdot \P(\tau^Y_1 < S) + (\E[N(\tau^A_1)] +\E[N(\tau^Y_1)])\P(\tau^Y_1 \geq S)  \leq 2 \frac{ e^{-\lambda}}{e^{\lambda}+e^{-\lambda}} \E[N(\tau^Y_1)]^2 .\]
    But \[v_A(\lambda,\mu) - \hat{v}(\lambda, \mu) = \frac{\E[A_{\tau^A_1}]}{\E[\tau_1^A]}- \frac{\E[Y_{\tau^Y_1}]}{\E[\tau_1^Y]}\] and by construction $\E[\tau_1^A] = \E[\tau_1^Y]$ so, 
    \[v_A(\lambda,\mu) - \hat{v}(\lambda, \mu) = \frac{\E[A_{\tau^A_1}-  Y_{\tau^Y_1}]}{\E[\tau_1^Y]} \leq \frac{2\E[N(\tau^Y_1)]^2}{\E[\tau_1^Y]} \frac{ e^{-\lambda}}{e^{\lambda}+e^{-\lambda}},  \]
    This way we have that $ v_A(\mu) - \hat{v}(\lambda, \mu) \leq C_{\mu} e^{-2\lambda}$ with $C_{\mu} = \frac{2\E[N(\tau^Y_1)]^2}{\E[\tau_1^Y]}$ that depends on $\mu$. Note that as $\E[N(\tau_1^Y)] = e^{\kappa/\mu} $, and $\E[\tau_1^Y] = \frac{1}{\kappa} e^{\kappa/\mu} $,we have that $C_{\mu}$ is decreasing in $\mu$.\bigskip\\
    Further we can couple $(A_t)_{t\geq 0 }$ to $(Y_t)_{t\geq 0}$ on the same dynamical environment in the following way: 
    Let $\mathcal{P}$ be as before, then for $t \in \mathcal{P}$ let $U_t \sim Unif[0, e^{\lambda} + e^{-\lambda}]$ and $V_t \sim Unif[0, \kappa]$.
    \begin{itemize}
        \item If $U_t \in [0, e^{\lambda})$ then $Y$ attempts a jump in direction $e_1$ it succeeds if $V_t \in [0, \omega_t({Y_{t^-}},Y_{t^-}+e_1)]$.
        \item If $U_t \geq e^{\lambda} $ then $Y$ attempts a jump in direction $-e_1$ it succeeds if $V_t \in [0, \omega_t({Y_{t^-}},Y_{t^-}-e_1)]$.
        \item $A$ always attempts a jump in direction $e_1$ it succeeds if $V_t \in [0, \omega_t({A_{t^-}},A_{t^-}+e_1)]$.
    \end{itemize}
    We then have $A_t\cdot e_1 \geq Y_t \cdot e_1$ for all $t \geq 0$, so
    \[
        v_A(\mu) - \hat{v}(\lambda, \mu) = \lim_{t \rightarrow \infty} \frac{A_t}{t} - \frac{Y_t}{t} \geq 0 \quad \text{a.s.}
    \]
\end{proof}
\begin{cor} \label{cor_v1} Let $d=1$. 
    Let $\lambda >0$, $\mu> 0$, and $\tilde{\omega} \sim q$, then for a constant $C_{\mu}$ that only depends on $\mu$ we have 
    \[0 \leq (e^{\lambda} + e^{-\lambda})  \frac{\E\left[\frac{\tilde{\omega}}{\mu + \tilde{\omega}}\right]}{\E\left[\frac{1}{\mu + \tilde{\omega}}\right]} - v(\lambda,\mu(e^{\lambda} + e^{-\lambda})) \leq C_{\mu} e^{-\lambda}.\]
\end{cor}
\begin{proof}
    Use Lemma \ref{speedNVvsV} \[v(\lambda,\mu(e^{\lambda} + e^{-\lambda})) =Z_{\lambda} \hat{v}(\lambda, \mu) =(e^\lambda + e^{-\lambda}) \hat{v}(\lambda, \mu), \] 
    then using the previous Lemma \ref{lem14} yields the claim. Note that the constant $C_{\mu}$ is here the same as in Lemma \ref{lem14}.
\end{proof}
\subsection{Coupling a high dimensional VBRW to a one-dimensional VBRW}
Let $d \geq 2$ and $m = \E[\tilde{\omega}]$, for $\tilde{\omega} \sim q$.\\
We want to construct$(\omega_t,X_t)_{t \geq 0}$ be a VBRW on dynamical conductances on $\zd$ with parameter  $\lambda, \mu >0$, and $(w_t,Y_t)_{t \geq 0}$ be a VBRW on dynamical conductances on $\mathbb{Z}$ with parameter  $\lambda, \tilde{\mu} = \mu + m(2d-2) >0$. Let $X_0 = Y_0 = 0$, $\omega_0$ is distributed according to $\Q$, and $\omega_t$ refreshes on each edge independently at rate $\mu$ and according to the measure $q$.\\
The edges to the left of $0$ refresh in $w_t$ according to a Poisson process at rate $\tilde{\mu}$ and according to $q$.
Let $\mathcal{P}_1,\mathcal{P}_2,\mathcal{P}_3$ be $3$ independent Poisson point processes with respective parameter $e^{\lambda}\kappa$, $e^{-\lambda} \kappa$ and $(2d-2)\kappa$.\\
\begin{enumerate}
    \item At points of $\mathcal{P}_1$ both $X$ and $Y$ attempt a jump in direction $e_1$, and we add the corresponding edge to the infected set $(I^X_t)_{t \geq 0}$ and $(I^Y_t)_{t \geq 0}$. We then say that the 2 copies of the edge are a match.
    \item At points of $\mathcal{P}_2$ both $X$ and $Y$ attempt a jump in direction $-e_1$, and we add the corresponding edge to the infected set $(I^X_t)_{t \geq 0}$ and $(I^Y_t)_{t \geq 0}$. We then say that the 2 copies of the edge are a match.
    \item At points of $\mathcal{P}_3$ $X$ attempts a jump in of the $(2d-2)$ other directions, and we add the edge only to the infected set $(I^X_t)_{t \geq 0}$.
\end{enumerate}
Let $(T_i)_{i \in \mathbb{N}}$ be the points of $\mathcal{P}_3$ and let $S$ be the first point of $\mathcal{P}_2$. We stop the coupling at time $T_2 \wedge S$.\bigskip \\
Now we want to explain how to remove copies of edges in the infected sets.\\
When we pick a copy of an edge to be removed of $(I^X_t)_{t \geq 0}$ (following the definition of an infected set), we also remove its match in $(I^Y_t)_{t \geq 0}$. Then we updated the conductances in $(\omega_t)_{t\geq 0}$ and $(w_t)_{t \geq 0}$ as in the definition of the infected sets.\bigskip \\
We next focus on coupling the environments $(\omega_t)_{t \geq 0}$ and $(w_t)_{t \geq 0}$.
Let $E(e)$ be the first time the edge $e$ is examined by $Y$ and $C(e)$ the fist time it is crossed by $Y$. 
\begin{itemize}
    \item If $E(e) \leq T_1 \wedge S$, then for all $s \in [E(e),C(e)\wedge T_1 \wedge S)$ we set $w_s(e) = \omega_s(e)$.
    \item If $E(e) \in (T_1,T_2 \wedge S)$, then for all $s \in [E(e),C(e)\wedge T_2 \wedge S)$ we set $w_s(e) = \omega_s(X_S + e_1)$.
    \item If $E(e) \leq T_1 \wedge S$ and $C(e) > T_1 \wedge S$, then for all $s \in [E(e),C(e))$ we set $w_s(e) = \omega_s(e)$ and for all $s \in [E(e),C(e)\wedge T_2 \wedge S)$ we set $w_s(e) = \omega_s(X_s,X_s + e_1)$.
    \item For $s \in (C(e)\wedge T_2 \wedge S,T_2 \wedge S)$, we refresh the edge $e$ in the environment $(w_t)_{t\geq 0}$ also at the points of a Poisson process with rate $m(2d-2)$, these updates do not affect the infected set.
\end{itemize}
After when we say that we stop the coupling, we let $(X_t)_{t\geq 0}$ and $(Y_t)_{t\geq 0}$ attempt jumps in $\mathcal{P}_1$, $\mathcal{P}_2$, and $\mathcal{P}_3$ as previously and each copy of edges $e_{i,j}$ in $(I^Y_t)_{t \geq 0}$ also refreshes at the points of a Poisson process $\mathcal{\tilde{P}}_{i,j}$ with rate $m(2d-2)$. If $e_{i,j}$ get refreshed at $t \in \mathcal{\tilde{P}}_{i,j}$  we do not remove it from $I^Y_t$ but if $j=1$ we resample the value of $w_t(e_i)$.\\

Let $(\tau_i)_{i \geq 0}$ be the successive times at which $I^X$ becomes $\emptyset$, then because of how we have constructed $(I^X_t)_{t\geq0}$ and $(I^Y_t)_{t\geq0}$ we have that $I^Y_{\tau_i} = \emptyset$ for all $i \in \mathbb{N}$.

\begin{rmk}
    Note that in this coupling once an edge is examined by $Y$ it refreshes at rate $\tilde{\mu}$. Indeed, up to $T_1$ (that happens at rate $2d-2$) the edge updates at rate $\mu$. Now suppose at time $T_1$ $X$ attempts a jump in direction $e_i$ then it succeeds with probability $\omega_{T_1}(X_{T_1^-},X_{T_1^-}+e_i)$, so under $\P$ the $X$ jumps in another direction then $e_1, -e_1$ at rate $m(2d-2)$. If for an edge $e$ (and there will be exactly one) $X$ jumps in another direction then $e_1, -e_1$ between $E(e) \wedge S$ and $C(e) \wedge S$ then $w(e)$ updates according $q$ as the value $\omega_{T_1}(X_{T_1}, X_{T_1} + e_1)$ is still unknown and distributed according to $q$. So after an edge is examined in $w$ it updates at rate $\tilde{\mu}$.\\
    Further, using the regeneration sequence $(\tau_i)_{i \geq 0}$ we get $v^Y$ the speed of $Y$ as 
    \[v^Y(\lambda,\tilde{\mu}) = \frac{\E[Y_{\tau_1}]}{\E[\tau_1]}.\]
\end{rmk}

\begin{lem} \label{Lemma_dto1}
    Let $\mu >0$ then for all $\lambda >0$ we have that 
    \[|v^Y(\lambda, Z_{\lambda}\mu + m(2d-2)) - v(\lambda, Z_{\lambda}\mu) | \leq \tilde{C}_{\mu} e^{-\lambda},\]
    with $\tilde{C}_{\mu}$ a constant that only depends on $\mu$.
\end{lem}

\begin{proof}
    Let $(\omega_t,X_t)_{t \geq 0}$ be a VBRW on dynamical conductances on $\zd$ with parameter  $\lambda, \Zl \mu >0$, and $(w_t,Y_t)_{t \geq 0}$ be a VBRW on dynamical conductances on $\mathbb{Z}$ with parameter  $\lambda, \tilde{\mu} = \Zl \mu + m(2d-2) >0$. We couple $(\omega_t,X_t)_{t \geq 0}$ and $(w_t,Y_t)_{t \geq 0}$ as above.
    Take \[A = \{S < \tau_1\} \cup \{T_2 < \tau_1\}.\]
    Then we have that 
    \[|v^Y(\lambda, Z_{\lambda}\mu + m(2d-2)) - v(\lambda, Z_{\lambda}\mu) | \leq \frac{1}{\E[\tau_1]} \E[|X_{\tau_1}\cdot e_1 - Y_{\tau_1}| \mathds{1}_A].\]
    Let $N(t) = |\{0\leq s \leq t: s \in \mathcal{P}_1 \cup \mathcal{P}_2 \cup \mathcal{P}_3 \}|$, then we get
    \[\E[|X_{\tau_1}\cdot e_1 - Y_{\tau_1}| \mathds{1}_A] \leq 2 \E[\mathds{1}_A N(\tau_1)].\]
    Further we have that the number of points in $\mathcal{P}_2[0,\tau_1]$ is binomial with parameter $(N(\tau_1), e^{-\lambda}Z_{\lambda}^{-1})$ and the number of points in $\mathcal{P}_3[0,\tau_1]$ is binomial with parameter $(N(\tau_1), (2d-2)Z_{\lambda}^{-1})$, we then have that 
    \begin{eqnarray*}
        \E[\mathds{1}_A N(\tau_1)] &\leq& \E[N(\tau_1)\mathds{1}_{\{S < \tau_1\}} ] + \E[N(\tau_1)\mathds{1}_{\{T_2 < \tau_1\}} ]\\
        &=& \E\left[N(\tau_1)(1-(1-e^{-\lambda}Z_{\lambda}^{-1})^{N(\tau_1)})\right] \\
        &&\quad + \E\left[N(\tau_1)\left( (1-(1-(2d-2)Z_{\lambda}^{-1})^{N(\tau_1)})  + N(\tau_1)(2d-2) Z_{\lambda}^{-1}(1-(2d-2) Z_{\lambda}^{-1})^{N(\tau_1)-1}\right) \right],
    \end{eqnarray*}
    with $(1-x)^a \geq 1-ax$ for all $a \in \mathbb{N}$ and $x\in(0,1)$ we can simplify this to
    \begin{eqnarray*}
        \E[\mathds{1}_A N(\tau_1)] &\leq& \E\left[ N(\tau_1)^2e^{-\lambda}Z_{\lambda}^{-1} +N(\tau_1) (N(\tau_1) -1) (2d-2)^2 Z_{\lambda}^{-2}\right] \leq 2d e^{-2\lambda} \E[N(\tau_1)^2].
    \end{eqnarray*}
    But by Lemma \ref{finiteTau} we have that $\E[N(\tau_1)^2] < \infty$. Further, as the $|I_t|_{t\geq0}$ has birth rate $\kappa \Zl$ and death rate $\mu \Zl |I_t|$ this way $N(\tau_1)$ is not depending on $\lambda$, and $\E[N(\tau_1)^2]$ is a constant depending on $\mu$.\\
    Next we look at $\E[\tau_1]$ which depends on $\lambda$ so in the following we will write $\tau_1^{\lambda}$ for the regeneration time corresponding to the process with parameter $\lambda$ and $\mu \Zl$. We then have that \[\E[\tau_1^{\lambda}] = Z_{\lambda}^{-1} \E[\tau_1^{1}],\]
    where we have that $\E[\tau_1^{1}]$ is a finite constant depending on $\mu$. \\
    Altogether we get  for some constant $\tilde{C}_{\mu} = \frac{\E[N(\tau_1)^2]}{\E[\tau_1^{1}]} $ depending on $\mu$ 
    \[|v^Y(\lambda, Z_{\lambda}\mu + m(2d-2)) - v(\lambda, Z_{\lambda}\mu) | \leq \frac{1}{\E[\tau_1]} \E[|X_{\tau_1}\cdot e_1 - Y_{\tau_1}| \mathds{1}_A] \leq \tilde{C}_{\mu} e^{-\lambda}.\]
\end{proof}

\begin{thm}\label{assymptotics}
    Let $\lambda >0$, $\mu > 0$ and $\tilde{\omega}$ a random variable distributed according to $q$, then 
    \[v(\lambda, Z_{\lambda}\mu) = e^{\lambda}\frac{\E\left[\frac{\tilde{\omega}}{\mu + \tilde{\omega}}\right]}{\E\left[\frac{1}{\mu + \tilde{\omega}}\right]} + (2d-2) \frac{\E\left[\frac{\tilde{\omega}(\tilde{\omega} - m)}{(\mu + \tilde{\omega})^2}\right]\E\left[ \frac{1}{\mu + \tilde{\omega}}\right] -\E\left[ \frac{\tilde{\omega}-m}{(\mu + \tilde{\omega})^2}\right] \E\left[\frac{\tilde{\omega}}{\mu + \tilde{\omega}}\right]}{\E\left[\frac{1}{\mu + \tilde{\omega}}\right]^2} + O(e^{-\lambda}) ,\]
    with $m = \E[\tilde{\omega}]$ as before.
\end{thm}
\begin{proof}
    First note that the process $Y$ is constructed such that \[v^Y(\lambda, Z_{\lambda}\mu + m(2d-2)) = v^1(\lambda,Z_{\lambda}\mu + m(2d-2)).\] 
    This is because starting at the time an edge is examined, it will refresh at rate $Z_{\lambda}\mu + m(2d-2)$ and so the process will behave like a one-dimensional random walk on dynamical conductances with parameter $\lambda$ and $Z_{\lambda}\mu + m(2d-2)$. Before the time an edge is examined it may refresh at rate $Z_{\lambda}\mu$, but that is irrelevant as we still have that the value of the conductance at the first examination time is distributed according to $q$.\\
    Using Lemma \ref{Lemma_dto1} we have that for $\tilde{C}_{\mu}$ a constant depending on $\mu$ we have 
    \[|v^1(\lambda, Z_{\lambda}\mu + m(2d-2)) - v(\lambda, Z_{\lambda}\mu) | \leq \tilde{C}_{\mu} e^{-\lambda}.\]
    But \[v^1\left(\lambda,Z_{\lambda}\mu + m(2d-2)\right) = v^1\left(\lambda,(e^{\lambda} + e^{-\lambda})\cdot \frac{Z_{\lambda}}{e^{\lambda} + e^{-\lambda}}\left(\mu + m\frac{2d-2}{Z_{\lambda}}\right)\right).\] 
    Using Corollary \ref{cor_v1} 
    \[0 \leq (e^{\lambda} + e^{-\lambda})  \frac{\E\left[\frac{\tilde{\omega}}{\frac{Z_{\lambda}}{e^{\lambda} + e^{-\lambda}}\left(\mu + m\frac{2d-2}{Z_{\lambda}}\right)+ \tilde{\omega}}\right]}{\E\left[\frac{1}{\frac{Z_{\lambda}}{e^{\lambda} + e^{-\lambda}}\left(\mu + m\frac{2d-2}{Z_{\lambda}}\right) + \tilde{\omega}}\right]} - v^1\left(\lambda,Z_{\lambda}\mu + m(2d-2)\right) \leq C_{\frac{Z_{\lambda}}{e^{\lambda} + e^{-\lambda}}\left(\mu + m\frac{2d-2}{Z_{\lambda}}\right)} e^{-\lambda}.\]
    Let $\delta = \delta(\lambda) = \frac{2d-2}{Z_{\lambda}}$ then we have $\delta(\lambda) \rightarrow 0$ for $\lambda \rightarrow \infty$ and recalling that $C_{\mu}$ is decreasing in $\mu$ we can rewrite the equation above as 
    \[0 \leq (2d-2)\frac{1-\delta}{\delta}  \frac{\E\left[\frac{\tilde{\omega}}{\frac{\mu+m\delta}{1-\delta}+ \tilde{\omega}}\right]}{\E\left[\frac{1}{\frac{\mu+m\delta}{1-\delta} + \tilde{\omega}}\right]} - v^1\left(\lambda,\left((2d-2)\frac{1-\delta}{\delta}\right)\left(\frac{\mu+m\delta}{1-\delta}\right)\right) \leq C_{\mu} e^{-\lambda}.\]
    Then the triangular inequality yields
    \[\left|v(\lambda, Z_{\lambda}\mu) - (2d-2)\frac{1-\delta}{\delta}  \frac{\E\left[\frac{\tilde{\omega}}{\frac{\mu+m\delta}{1-\delta}+ \tilde{\omega}}\right]}{\E\left[\frac{1}{\frac{\mu+m\delta}{1-\delta} + \tilde{\omega}}\right]} \right| \leq (\tilde{C}_{\mu} + C_{\mu}) e^{-\lambda},\]
    so
    \[v(\lambda, Z_{\lambda}\mu) =   (2d-2)\frac{1-\delta}{\delta}  \frac{\E\left[\frac{\tilde{\omega}}{\frac{\mu+m\delta}{1-\delta}+ \tilde{\omega}}\right]}{\E\left[\frac{1}{\frac{\mu+m\delta}{1-\delta} + \tilde{\omega}}\right]}  + O(e^{-\lambda}).\]
    We now Taylor expand the expression $(1-\delta)\frac{\E\left[\frac{\tilde{\omega}}{\frac{\mu+m\delta}{1-\delta}+ \tilde{\omega}}\right]}{\E\left[\frac{1}{\frac{\mu+m\delta}{1-\delta} + \tilde{\omega}}\right]}$ in $\delta \rightarrow 0$ using that we may exchange the expectation and differentiation in $\delta$ by monotone convergence.
    \begin{equation}\label{AssymptoticsRescalled}
        (1-\delta)\frac{\E\left[\frac{\tilde{\omega}}{\frac{\mu+m\delta}{1-\delta}+ \tilde{\omega}}\right]}{\E\left[\frac{1}{\frac{\mu+m\delta}{1-\delta} + \tilde{\omega}}\right]} = 
        \frac{\E\left[\frac{\tilde{\omega}}{\mu + \tilde{\omega}}\right]}{\E\left[\frac{1}{\mu + \tilde{\omega}}\right]}
        + \frac{\E\left[\frac{\tilde{\omega}(\tilde{\omega} - m)}{(\mu + \tilde{\omega})^2}\right]\E\left[ \frac{1}{\mu + \tilde{\omega}}\right] -\E\left[ \frac{\tilde{\omega}-m}{(\mu + \tilde{\omega})^2}\right] \E\left[\frac{\tilde{\omega}}{\mu + \tilde{\omega}}\right] - \E\left[\frac{\tilde{\omega}}{\mu + \tilde{\omega}}\right] \E\left[\frac{1}{\mu + \tilde{\omega}}\right]}{\E\left[\frac{1}{\mu + \tilde{\omega}}\right]^2} \delta + O(\delta^2).
    \end{equation}

    Now using that $O(\delta^2) = O(e^{-2\lambda})$ we have that, 
    \begin{eqnarray*}
        v(\lambda, Z_{\lambda}\mu) &=& Z_{\lambda} \frac{\E\left[\frac{\tilde{\omega}}{\mu + \tilde{\omega}}\right]}{\E\left[\frac{1}{\mu + \tilde{\omega}}\right]} + \frac{2d-2}{\delta} \delta \frac{\E\left[\frac{\tilde{\omega}(\tilde{\omega} - m)}{(\mu + \tilde{\omega})^2}\right]\E\left[ \frac{1}{\mu + \tilde{\omega}}\right] -\E\left[ \frac{\tilde{\omega}-m}{(\mu + \tilde{\omega})^2}\right] \E\left[\frac{\tilde{\omega}}{\mu + \tilde{\omega}}\right] - \E\left[\frac{\tilde{\omega}}{\mu + \tilde{\omega}}\right] \E\left[\frac{1}{\mu + \tilde{\omega}}\right]}{\E\left[\frac{1}{\mu + \tilde{\omega}}\right]^2} + O(e^{-\lambda})\\
        &=& Z_{\lambda} \frac{\E\left[\frac{\tilde{\omega}}{\mu + \tilde{\omega}}\right]}{\E\left[\frac{1}{\mu + \tilde{\omega}}\right]} + (2d-2) \frac{\E\left[\frac{\tilde{\omega}(\tilde{\omega} - m)}{(\mu + \tilde{\omega})^2}\right]\E\left[ \frac{1}{\mu + \tilde{\omega}}\right] -\E\left[ \frac{\tilde{\omega}-m}{(\mu + \tilde{\omega})^2}\right] \E\left[\frac{\tilde{\omega}}{\mu + \tilde{\omega}}\right] - \E\left[\frac{\tilde{\omega}}{\mu + \tilde{\omega}}\right] \E\left[\frac{1}{\mu + \tilde{\omega}}\right]}{\E\left[\frac{1}{\mu + \tilde{\omega}}\right]^2} + O(e^{-\lambda})\\
        &=& e^{\lambda}\frac{\E\left[\frac{\tilde{\omega}}{\mu + \tilde{\omega}}\right]}{\E\left[\frac{1}{\mu + \tilde{\omega}}\right]} + (2d-2) \frac{\E\left[\frac{\tilde{\omega}(\tilde{\omega} - m)}{(\mu + \tilde{\omega})^2}\right]\E\left[ \frac{1}{\mu + \tilde{\omega}}\right] -\E\left[ \frac{\tilde{\omega}-m}{(\mu + \tilde{\omega})^2}\right] \E\left[\frac{\tilde{\omega}}{\mu + \tilde{\omega}}\right]}{\E\left[\frac{1}{\mu + \tilde{\omega}}\right]^2} + O(e^{-\lambda}).\\
    \end{eqnarray*} 
\end{proof}
\begin{cor}\label{corAssymtoticNVBRW}
    Let $\lambda >0$, $\mu > 0$ and $\tilde{\omega}$ a random variable distributed according to $q$ and $m = \E[\tilde{\omega}]$ as before.
    \[\hat{v}(\lambda,\mu) = \frac{\E\left[\frac{\tilde{\omega}}{\mu + \tilde{\omega}}\right]}{\E\left[\frac{1}{\mu + \tilde{\omega}}\right]}
    + (2d-2)\frac{\E\left[\frac{\tilde{\omega}(\tilde{\omega} - m)}{(\mu + \tilde{\omega})^2}\right]\E\left[ \frac{1}{\mu + \tilde{\omega}}\right] -\E\left[ \frac{\tilde{\omega}-m}{(\mu + \tilde{\omega})^2}\right] \E\left[\frac{\tilde{\omega}}{\mu + \tilde{\omega}}\right] - \E\left[\frac{\tilde{\omega}}{\mu + \tilde{\omega}}\right] \E\left[\frac{1}{\mu + \tilde{\omega}}\right]}{\E\left[\frac{1}{\mu + \tilde{\omega}}\right]^2} e^{-\lambda} + O(e^{-2\lambda}). \]
\end{cor}
\begin{proof}
    First note that $\hat{v}(\lambda,\mu) = \frac{v(\lambda, Z_{\lambda}\mu)}{Z_{\lambda}}$ in the notation of the previous proof we get that \[\hat{v}(\lambda,\mu) =\frac{\delta v(\lambda, Z_{\lambda}\mu)}{2d-2}.\]
    Then we get the claim using equation \eqref{AssymptoticsRescalled}.
\end{proof}
As last thing in this section we will give an alternative representation of the first order term of the Taylor expression in Theorem \ref{assymptotics}. This representation will be useful later because it will simplify some computations.
\begin{lem}
    \begin{eqnarray*}
        &&\frac{\E\left[\frac{\tilde{\omega}(\tilde{\omega} - m)}{(\mu + \tilde{\omega})^2}\right]\E\left[ \frac{1}{\mu + \tilde{\omega}}\right] -\E\left[ \frac{\tilde{\omega}-m}{(\mu + \tilde{\omega})^2}\right] \E\left[\frac{\tilde{\omega}}{\mu + \tilde{\omega}}\right] - \E\left[\frac{\tilde{\omega}}{\mu + \tilde{\omega}}\right] \E\left[\frac{1}{\mu + \tilde{\omega}}\right]}{\E\left[\frac{1}{\mu + \tilde{\omega}}\right]^2} \\
        &=& (m+\mu)\frac{\E\left[\frac{1}{(\mu + \tilde{\omega})^2}\right]\E\left[ \frac{\tilde{\omega}}{\mu + \tilde{\omega}}\right] -\E\left[ \frac{\tilde{\omega}}{(\mu + \tilde{\omega})^2}\right] \E\left[\frac{1}{\mu + \tilde{\omega}}\right] }{\E\left[\frac{1}{\mu + \tilde{\omega}}\right]^2} - \frac{\E\left[\frac{\tilde{\omega}}{\mu + \tilde{\omega}}\right]}{\E\left[\frac{1}{\mu + \tilde{\omega}}\right]}.
    \end{eqnarray*}
\end{lem}

\begin{rmk}
    We see an important change of behavior in the last line regarding whether the process is a normalized time change of the variable speed process, compared to the real variable speed process. The second does not have the negative term anymore. It seems this way that the normalized variable speed process and the variable speed process behave asymptotically in a fundamentally different way.
\end{rmk}

\section{Monotonicity of the speed of NVBRW}
\subsection{Monotonicity of VBRW in $d=1$}
    \begin{thm}\label{Monotonmularge}
        Let $\lambda >0$ and $\varepsilon >0$ then for all $\mu >0 $ 
        \[v^1(\lambda,\mu)\leq v^1(\lambda + \varepsilon,\mu).\]
    \end{thm}
    \begin{proof}
        Let $(\omega_t)_{t \geq 0}$ be dynamical conductances on $\mathbb{Z}$. Then we construct a coupling using the alternative representation of the processes, between a VBRW $(X_t^{\lambda})_{t\geq0}$ with parameter $\lambda,\mu$ and a VBRW $(X_t^{\lambda+\varepsilon})_{t\geq0}$ with parameter $\lambda+\varepsilon,\mu$, such that for all $t \geq 0$, $X_t^{\lambda +\varepsilon} \geq X_t^{\lambda}$.
        Take $\mathcal{P}$ to be the points of a Poisson process with rate $\kappa (e^{\lambda + \varepsilon}+e^{-\lambda})$. Then for $t \in \mathcal{P}$ we sample an independent uniform random variable $U \sim \text{Unif}([0,(e^{\lambda + \varepsilon}+e^{-\lambda})])$.
        \begin{itemize}
            \item If $U \in [0, e^{\lambda})$ then both processes attempt a jump to the right (direction $e_1$). 
            \item If $U \in [e^{\lambda}, e^{\lambda + \varepsilon})$ then only $X^{\lambda + \varepsilon}$ attempts a jump to the right (direction $e_1$).
            \item If $U \in [e^{\lambda + \varepsilon}, (e^{\lambda + \varepsilon} + e^{-(\lambda + \varepsilon)}))$ then both processes attempt a jump to the left (direction $-e_1$).
            \item If $U \geq (e^{\lambda + \varepsilon} + e^{-(\lambda + \varepsilon)})$ then only $X^{\lambda}$ attempts a jump to the left (direction $-e_1$).
        \end{itemize}
        Then sample an independent uniform random variable $V \sim \text{Unif}([0,\kappa])$.\\
        \begin{itemize}
            \item If $X^{\lambda}$ is attempting a jump in direction $e$ it succeeds if and only if $ V \leq \omega_t(X^{\lambda}_t,X^{\lambda}_t + e)$.
            \item If $X^{\lambda + \varepsilon}$ is attempting a jump in direction $e$ it succeeds if and only if $ V \leq \omega_t(X^{\lambda+ \varepsilon}_t,X^{\lambda+ \varepsilon}_t + e)$.
        \end{itemize}
        In this coupling we have that $X^{\lambda}_t \leq X^{\lambda+\varepsilon}_t$ for all $t \geq 0$ a.s.
    \end{proof}
\subsection{A coupling between NVBRW with different biases} \label{couplingDiffBias}
    Let $\kappa > 0 $ and $q$ a measure on $([0,\kappa],\mathcal{B}([0,\kappa]))$ and $Q = q^{E(\zd)}$ the product measure of $q$ on all the edges of $\zd$. In the following we will assume that the conductances in our models will all be distributed according to $Q$.\bigskip \\
    Let $(\eta_t,X_t)_{t\geq 0}$ be a NVBRW on dynamical conductances with parameter $\lambda > 0$ and $\mu > 0$. Further, for $\varepsilon >0$ we let $(\nu_t,Y_t)_{t\geq 0}$ be a NVBRW on dynamical conductances with parameter $\lambda + \varepsilon$ and $\mu$. The goal of this section is to show that for $\mu$ large enough $\hat{v}(\lambda + \varepsilon,\mu) \geq \hat{v}(\lambda,\mu)$.\bigskip\\
    Let $\omega$ be distributed according to $q$, then $m = \E[\omega]$. Recall that by our assumptions $m >0$. \bigskip\\
    We now want to couple $(\eta_t,X_t)$ with $(\nu_t,Y_t)$. We denote by $(I^X_t)_{t \geq 0}$ the infected set of $X$ and by $(I^Y_t)_{t \geq 0}$ the infected set of $Y$. The coupling is similar to the coupling presented in the case of dynamical percolation in \cite{percolation}. \\
    Let $X_0 =Y_0=0$ and $\mathcal{P}$ be a Poisson process with rate $\kappa$. The points of $\mathcal{P}$ are the points at which both process will attempt jumps. More precisely for $t \in \mathcal{P}$ let $U$ be a uniform random variable on $[0,1]$
    \begin{itemize}
        \item[(1)] If $U < (2d-2)/Z_{\lambda+\varepsilon}$ then $X$ and $Y$ both attempt a jump in one of the $2d-2$ direction $e$ other than $e_1,-e_1$. (The direction is then chosen uniformly at random).
        \item[(2)] If $U\in[(2d-2)/Z_{\lambda+\varepsilon},(2d-2)/Z_{\lambda})$ then $X$ attempts a jump in one of the $2d-2$ direction $e$ other than $e_1,-e_1$ and $Y$ attempts a jump on direction $e_1$.
        \item[(3)] If $U\in[(2d-2)/Z_{\lambda},(2d-2)/Z_{\lambda}+e^{-(\lambda+\varepsilon)}/Z_{\lambda+\varepsilon})$ then both process attempt a jump in direction $-e_1$.
        \item[(4)] If $U\in[(2d-2)/Z_{\lambda}+e^{-(\lambda+\varepsilon)}/Z_{\lambda+\varepsilon},1-e^{\lambda})$ then $X$ attempts a jump in direction $-e_1$ and $Y$ attempts a jump in direction $e_1$.
        \item[(5)] If $U \geq 1-e^{\lambda}$ then $X$ and $Y$ both attempt a jump in direction $e_1$.
    \end{itemize}
    Next we split the points on $\mathcal{P}$ into three groups, to $t\in \mathcal{P}$ we say that 
    \begin{itemize}
        \item $t$ is a \textbf{good point} if in $t$ case (5) occurs, we denote the Poisson process of good points by $\mathcal{P}_g$, it has rate $r_g = \kappa e^{\lambda}/Z_{\lambda}$,
        \item $t$ is a \textbf{bad point} if in $t$ case (1) or (3) occurs, we denote the Poisson process of bad points by $\mathcal{P}_b$, it has rate $r_b = \kappa (2d-2+e^{-(\lambda+\varepsilon)})/Z_{\lambda+\varepsilon}$,
        \item $t$ is a \textbf{very bad point} if in $t$ case (2) or (4) occurs, we denote the Poisson process of very bad points by $\mathcal{P}_{v}$, it has rate $r_{v} = \kappa \left(e^{\lambda+\varepsilon}/Z_{\lambda+\varepsilon} -e^{\lambda}/Z_{\lambda} \right)$.
    \end{itemize}
    In the following we will write $(T_i)_{i\geq 1}$ for the points of $\mathcal{P}$, $(T^g_i)_{i\geq 1}$ for the points of $\mathcal{P}_g$, $(T^b_i)_{i\geq 1}$ for the points of $\mathcal{P}_b$, $(T^v_i)_{i\geq 1}$ for the points of $\mathcal{P}_v$
    Now assume that at time $t$ $X$ attempts a jump in direction $e$ and $Y$ attempts a jump in direction $\tilde{e}$. We will also say that the edge $\{X_{t^-},X_{t^-} + e\}$ is examined by $X$ and the edge $\{Y_{t^-},Y_{t^-} + \tilde{e}\}$ is examined by $Y$. Let $V$ be a uniform random variable on $[0,\kappa]$.
    \begin{itemize}
        \item If $V \in [0,\eta_t(X_{t^{-}},X_{t^{-}}+e)]$ then $X_t = X_{t^{-}}+e$.
        \item If $V \in [0,\nu_t(Y_{t^{-}},Y_{t^{-}}+\tilde{e})]$ then $Y_t = Y_{t^{-}}+\tilde{e}$.
    \end{itemize}
    If it is the first time that $X$ examines the edge $\{X_{t^-},X_{t^-} + e\}$ then we refresh that edge such that $\eta_t(X_{t^{-}},X_{t^{-}}+e) = \nu_t(X_{t^{-}},X_{t^{-}}+e)$. Note that the value of $\eta$ on an edge only has an influence on the process $X$ starting at the time the edge is first examined, and that at the value we assigned is distributed according to $q$. This way we have that updating the value of the edge the first time we examine it is not changing the behavior of $X$.\bigskip \\
    Further $I^X_t = I^X_{t^-} \cup \{\{X_{t^-},X_{t^-} + e\}\}$ and $I^Y_t = I^Y_{t^-} \cup \{\{Y_{t^-},Y_{t^-} + \tilde{e}\}\}$. As we add points the at the same time to $I^X$ and $I^Y$ we can define $\mathcal{R}$ a Poisson process with rate function $ |I^X| \mu$ then at each point of the process $\mathcal{R}$ we remove a copy of an edge of $I^X$ as well as of $I^Y$, this way we can ensure that $|I^X| = |I^Y|$, so both process will have the same regeneration times $(\tau_i)_{i\geq 1}$. \bigskip \\
    We stop the coupling at $T^v_1$. After $T^v_1$ $X$ and $Y$ continue to attempt jumps at points of $\mathcal{P}$, but they choose the direction in which they attempt a jump independently. Further, for $t \in \mathcal{R}$ with $t<T^v_1$ we remove a copy of the same edge in $I^X$ and $I^Y$, but for $t \geq T^v_1$ we chose the edge to remove in $I^X$ independently of the edge removed in $I^Y$. Now for $t \in \mathcal{R}$ with $t<T^v_1$ assume that the edge $e$ has to be refreshed at time $t$ in $\eta$ then it also has to be refreshed in $\nu$, we then sample $\omega$ according to $q$ and set $\eta(e) = \nu(e) = \omega$. This way we achieve that $X$ and $Y$ have the impression of running on the same environment up to time $T^v_1$. For $t \geq T^v_1$ we refresh the edges of $\nu$ and $\eta$ independently.
\subsection{Monotonicity for $\mu$ large enough}
    \begin{thm}\label{monotonemularge}
        For $\lambda>0$ then there exists $M\geq 0$ such that for $\mu >M$, $\varepsilon >0$ we have that
        \[\hat{v}(\lambda + \varepsilon,\mu) \geq \hat{v}(\lambda,\mu).\]
    \end{thm}
    
    \begin{proof} 
        Let $(\eta_t,X_t)_{t \geq 0}$ be a NVBRW on dynamical conductances with parameters $\lambda >0$ and $\mu>0$. Let $\varepsilon >0$ and let $(\nu_t,Y_t)_{t\geq 0}$ be a NVBRW on dynamical conductances with parameter $\lambda + \varepsilon$ and $\mu$. We couple $(\eta_t,X_t)_{t \geq 0}$ and $(\nu_t,Y_t)_{t\geq 0}$ as above in subsection \ref{couplingDiffBias}.\bigskip \\
        Our goal is to show that $\hat{v}(\lambda + \varepsilon,\mu) \geq \hat{v}(\lambda,\mu) \geq 0$.
        To do so recall that \[\hat{v}(\lambda + \varepsilon,\mu) \geq \hat{v}(\lambda,\mu) = \frac{\E[Y_{\tau_1} \cdot e_1-X_{\tau_1}\cdot e_1]}{\E[\tau_1]},\] so it suffices to show that \[\E[Y_{\tau_1} \cdot e_1-X_{\tau_1}\cdot e_1] \geq 0.\]
        Let $N = \max\{n \in \N : T_n \leq \tau_1\}$ be the number of jump attempts up to time $\tau_1$. Recall by Lemma \ref{finiteTau} $N$ has exponential tails. Note that $N$ is depending on $\mu$, so we will write $N_{\mu}$ in the second part of this proof when we want to look at $\mu$ going to $\infty$, but for this first part we have $\mu$ fixed, so we will only write $N$.
        \[\E[Y_{\tau_1} \cdot e_1-X_{\tau_1}\cdot e_1] = \sum_{n=1}^{\infty} \E[Y_{\tau_1} \cdot e_1-X_{\tau_1}\cdot e_1|N=n] \cdot \P(N=n) .\]
        Let $n \geq 1$, by definition for all $t < T^v_1$ we have $X_t = Y_t$. \bigskip \\
        If $\tau_1 < T^v_1$ then $Y_{\tau_1} - X_{\tau_1} = 0$ so \[\E[Y_{\tau_1}\cdot e_1 - X_{\tau_1}\cdot e_1|\tau_1 < T^v_1, N=n] = 0\]
        and
        \[\E[Y_{\tau_1}\cdot e_1 - X_{\tau_1}\cdot e_1|N=n] = \E[Y_{\tau_1}\cdot e_1 - X_{\tau_1}\cdot e_1|\tau_1 \geq T^v_1,N=n] \P(\tau_1 \geq T^v_1|N=n).\]
        Let $V$ be that event such that we have exactly one very bad point up to time $\tau_1$, $V = \{\tau_1 \geq T^v_1\} \cap \{\tau_1 < T^v_2\}$. Then 
        \begin{eqnarray*}
            \E[Y_{\tau_1}\cdot e_1 - X_{\tau_1}\cdot e_1|N=n] &=& \E[Y_{\tau_1}\cdot e_1 - X_{\tau_1}\cdot e_1|V,N=n] \P(V|\tau_1 \geq T^v_1,N=n)\\
            &&\quad +\E[Y_{\tau_1}\cdot e_1 - X_{\tau_1}\cdot e_1|V^c,N=n] \P(V^c|\tau_1 \geq T^v_1,N=n)\\
            &\geq&\E[Y_{\tau_1}\cdot e_1 - X_{\tau_1}\cdot e_1|V,N=n] \P(V|\tau_1 \geq T^v_1,N=n)\\
            && \quad -2n\P(V^c|\tau_1 \geq T^v_1,N=n).\\
        \end{eqnarray*}
        We now compute the probabilities in the expression above. To do so we recall that the coloring (good, bad, very bad), is independent of $\mathcal{P}$ so also of $N$. This way the number of very bad points we see up to time $N$ is distributed Binomial$(N,\frac{r_v}{\kappa})$.
        \begin{eqnarray*}
            \P(V|\tau_1 \geq T^v_1,N=n) = n \frac{r_v}{\kappa} \left(1-\frac{r_v}{\kappa}\right)^{n-1}\\
        \end{eqnarray*}
        and
        \begin{eqnarray*}
            \P(V^c|\tau_1 \geq T^v_1,N=n) &=& \P(\tau_1 \geq T^v_1|N=n)-\P(V^c|\tau_1 \geq T^v_1,N=n)\\
            &=& 1 - \left(1-\frac{r_v}{\kappa}\right)^{n} - n \frac{r_v}{\kappa} \left(1-\frac{r_v}{\kappa}\right)^{n-1}\\
            &\leq& n \frac{r_v}{\kappa}- n \frac{r_v}{\kappa} \left(1-\frac{r_v}{\kappa}\right)^{n-1}\\
            &=& n \frac{r_v}{\kappa} \left(1-\left(1-\frac{r_v}{\kappa}\right)^{n-1}\right)\\
            &\leq&n(n-1) \left(\frac{r_v}{\kappa}\right)^2.\\
        \end{eqnarray*}
        Altogether we get:
        \begin{eqnarray*}
            \E[Y_{\tau_1}\cdot e_1 - X_{\tau_1}\cdot e_1|N=n] \geq n \frac{r_v}{\kappa} \left(1-\frac{r_v}{\kappa}\right)^{n-1}\E[Y_{\tau_1}\cdot e_1 - X_{\tau_1}\cdot e_1|V,N=n] -2n^2(n-1)\left(\frac{r_v}{\kappa}\right)^2.\\
        \end{eqnarray*}
        We now want to make a handle separately the events on which there are bad points before $\tau_1$ and the events in which there are none.
        \begin{eqnarray*}
            \E[Y_{\tau_1}\cdot e_1 - X_{\tau_1}\cdot e_1|V,N=n] &=& \P(T^b_1>\tau_1|V,N=n)\E[Y_{\tau_1}\cdot e_1 - X_{\tau_1}\cdot e_1|T^b_1>\tau_1,V,N=n]\\
            &&+ \P(T^b_1\leq\tau_1|V,N=n)\E[Y_{\tau_1}\cdot e_1 - X_{\tau_1}\cdot e_1|T^b_1\leq\tau_1,V,N=n] \\
            &=&\E[Y_{\tau_1}\cdot e_1 - X_{\tau_1}\cdot e_1|T^b_1>\tau_1,V,N=n]\\
            &&+ \P(T^b_1\leq\tau_1|V,N=n)\E[Y_{\tau_1}\cdot e_1 - X_{\tau_1}\cdot e_1|T^b_1\leq\tau_1,V,N=n]\\
            &&-\P(T^b_1\leq\tau_1|V,N=n)\E[Y_{\tau_1}\cdot e_1 - X_{\tau_1}\cdot e_1|T^b_1>\tau_1,V,N=n].\\
        \end{eqnarray*}
        Further using that on the events $V$ and $\{N=n\}$ the number of bad points has a Binomial distribution with parameter $n-1$ and $\frac{r_b}{\kappa}$
        \begin{eqnarray*}
            \P(T^b_1\leq\tau_1|V,N=n) &=& 1 - \P(T^b_1>\tau_1|V,N=n)\\
            &=& 1-\left(1- \frac{r_b}{\kappa}\right)^{n-1}\\
            &\leq&  (n-1)\frac{r_b}{\kappa}.
        \end{eqnarray*}
        Using this lower bound in the previous computation yields
        \begin{eqnarray*}
            \E[Y_{\tau_1}\cdot e_1 - X_{\tau_1}\cdot e_1|V,N=n] &\geq&\E[Y_{\tau_1}\cdot e_1 - X_{\tau_1}\cdot e_1|T^b_1>\tau_1,V,N=n]- 4n(n-1)\frac{r_b}{\kappa}.\\
        \end{eqnarray*}
        We now want to lower-bound $\E[Y_{\tau_1}\cdot e_1 - X_{\tau_1}\cdot e_1|T^b_1>\tau_1,V,N=n]$.\\
        To do so we will restrict ourselves to the case where at the very bad point $Y$ succeeds its jump in direction $e_1$.
        \begin{eqnarray*}
            &&\E[Y_{\tau_1}\cdot e_1 - X_{\tau_1}\cdot e_1|T^b_1>\tau_1,V,N=n] \geq \\
            &&\E[Y_{\tau_1}\cdot e_1 - X_{\tau_1}
            \cdot e_1|T^b_1>\tau_1,V,N=n,Y_{T^v_1} = Y_{(T^v_1)^-} + e_1] \cdot \P(Y_{T^v_1} = Y_{(T^v_1)^-} + e_1|T^b_1>\tau_1,V,N=n).
        \end{eqnarray*}
        Note that for all $i \geq 0 $ such that $T^v_1 < T_i \leq T_n$ on the event $\{T^b_1>\tau_1,V,N=n\}$ we have that $X$ and $Y$ are attempting jumps in direction $e_1$ at time $T_i$.\\
        Let $S = \{t \geq T^v_1: X_{t} \cdot e_1 \geq X_{T^v_1} \}$. We restrict ourselves to the event $\{T^b_1>\tau_1,V,N=n,Y_{T^v_1} = Y_{(T^v_1)^-} + e_1\}$ then for all $i,j \geq 0 $ such that $S < T_i \leq T_n$ and $T^v_1 < T_j \leq T_n$ we have by the construction of our coupling that the $\eta_{T_i}(X_{T_i^-},X_{T_i^-})$ is independent of $\nu_{T_j}(Y_{T_j^-},Y_{T_j^-})$. Further, for all $k \geq 0$ such that $T^v_1\leq T_k \leq S$ we have $X_{T_k} \cdot e_1 \leq Y_{T_k}\cdot e_1 + 1$, so we have that \[\E[Y_{\tau_1}\cdot e_1 - X_{\tau_1}\cdot e_1|T^b_1>\tau_1,V,N=n,Y_{T^v_1} = Y_{(T^v_1)^-} + e_1] \geq 1.\]
        Last we want to lower bound $\P(Y_{T^v_1} = Y_{(T^v_1)^-} + e_1|T^b_1>\tau_1,V,N=n)$. 
        Let $T_{before} = \max_{T_i < T^v_1: i \geq 0}$ and let $R$ be the event that $\omega((T^v_1)^-,(T^v_1)^-+e_1)$ 
        is resampled in the time interval $(T_{before},T^v_1)$ 
        then \[\P(R|T^b_1>\tau_1,V,N=n) = \frac{\mu}{\mu+\kappa}.\] 
        And for $V\sim Uniform[0,\kappa]$
        \begin{eqnarray*}
            \P(Y_{T^v_1} = Y_{(T^v_1)^-} + e_1|T^b_1>\tau_1,V,N=n) \geq \frac{\mu}{\mu+\kappa} \P(V \leq \omega_{T^v_1}((T^v_1)^-,(T^v_1)^-+e_1)) = \frac{\mu}{\mu+\kappa} \frac{m}{\kappa} > 0.
        \end{eqnarray*}
        So putting everything together we get:
        \begin{eqnarray*}
            \E[Y_{\tau_1}\cdot e_1 - X_{\tau_1}\cdot e_1|N=n] &\geq& n \frac{r_v}{\kappa} \left(1-\frac{r_v}{\kappa}\right)^{n-1}\E[Y_{\tau_1}\cdot e_1 - X_{\tau_1}\cdot e_1|V,N=n] -2n^2(n-1)\left(\frac{r_v}{\kappa}\right)^2\\
            &\geq& n \frac{r_v}{\kappa} \left(1-\frac{r_v}{\kappa}\right)^{n-1}\left[\E[Y_{\tau_1}\cdot e_1 - X_{\tau_1}\cdot e_1|T^b_1>\tau_1,V,N=n]- 4n(n-1)\frac{r_b}{\kappa}\right] \\
            &&\quad -2n^2(n-1)\left(\frac{r_v}{\kappa}\right)^2\\
            &\geq&n \frac{r_v}{\kappa} \left(1-\frac{r_v}{\kappa}\right)^{n-1}\left[\frac{\mu}{\mu+\kappa} \frac{m}{\kappa}- 4n(n-1)\frac{r_b}{\kappa}\right] -2n^2(n-1)\left(\frac{r_v}{\kappa}\right)^2\\
            &=&n \frac{r_v}{\kappa} \left(1-\frac{r_v}{\kappa}\right)^{n-1}\frac{\mu}{\mu+\kappa} \frac{m}{\kappa} -n^2(n-1)\left(\frac{r_v}{\kappa}\right)\left(4 \frac{r_b}{\kappa}\left(1-\frac{r_v}{\kappa}\right)^{n-1} +2\frac{r_v}{\kappa}\right)\\
            &\geq& n \frac{r_v}{\kappa} \left(1-\frac{r_v}{\kappa}\right)^{n-1}\frac{\mu}{\mu+\kappa} \frac{m}{\kappa} - n^2(n-1)\left(\frac{r_v}{\kappa}\right)\left(4 \frac{r_b}{\kappa}+2\frac{r_v}{\kappa}\right).\\
        \end{eqnarray*}
        So 
        \begin{eqnarray*}
            \E[Y_{\tau_1}\cdot e_1 - X_{\tau_1}\cdot e_1]&\geq&\frac{mr_v}{\kappa^2} \frac{\mu}{\mu+\kappa} \E\left[N(1-\frac{r_v}{\kappa})^{N-1}\right] -\left(\frac{r_v}{\kappa}\right)\left(4 \frac{r_b}{\kappa}+2\frac{r_v}{\kappa}\right) \E\left[N^2(N-1)\right].\\
        \end{eqnarray*}
        Now we want to show that for $\mu \rightarrow \infty$, $N_{\mu}$ converges in probability to $1$. Let $\varepsilon >0$ then
        \begin{eqnarray*}
            \P(|N_{\mu}-1| > \varepsilon) \leq \P(N_{\mu} \neq 1) = \frac{\kappa}{\mu+\kappa} \longrightarrow 0 \text{ for } {\mu \rightarrow \infty}.
        \end{eqnarray*}
        We can couple two discrete death birth chains $(W^1_n)_{n\in \N}$ and $(W^2_n)_{n\in \N}$ with respective parameters $\kappa,\mu_1$ and $\kappa,\mu_2$ and $\mu_1 > \mu_2$ such that the first return time of $(W^1_n)_{n\in \N}$ is smaller of equal to the first return time of $(W^2_n)_{n\in \N}$. This way we can assume that $(N_{\mu})_{\mu>0}$ is a decreasing sequence in $\mu$. Further the function $f(x) = x^2(x-1)$ is non-decreasing for $x > 2/3$, but $N_{\mu}\geq 1$ by definition, so by monotone convergence we get that 
        \begin{eqnarray*}
            \lim_{\mu \rightarrow \infty} \E\left[N^2(N-1)\right] = 0.
        \end{eqnarray*} 
        We also have that
        \begin{eqnarray*}
            \E[N_{\mu}(1-(N_{\mu}-1)\frac{r_v}{\kappa})] \leq \E\left[N_{\mu}(1-\frac{r_v}{\kappa})^{N_{\mu}-1}\right] \leq \E[N_{\mu}].
        \end{eqnarray*}
        As the function $g(x) = x(1-\frac{r_v}{\kappa}(x-1)) = (1+\frac{r_v}{\kappa})x -\frac{r_v}{\kappa} x^2$ is non-increasing for $x > \frac{2+\frac{r_v}{\kappa} }{2\frac{r_v}{\kappa}}$ so in particular for $x \geq 1$. Then by monotone convergence we get 
        \begin{eqnarray*}
            \lim_{\mu \rightarrow \infty}\E[N_{\mu}] = 1
        \end{eqnarray*}
        and
        \begin{eqnarray*}
            \lim_{\mu \rightarrow \infty}\E[N_{\mu}(1-(N_{\mu}-1)\frac{r_v}{\kappa})] = 1,
        \end{eqnarray*}
        so 
        \begin{eqnarray*}
            \lim_{\mu \rightarrow \infty} \E[Y_{\tau_1}\cdot e_1 - X_{\tau_1}\cdot e_1] \geq\frac{mr_v}{\kappa^2} > 0.
        \end{eqnarray*}
        For $\mu$ large enough we get the claim.
    \end{proof}

\subsection{Is $\mu$ large enough necessary?}
    In this subsection we show that even under the assumption that $q$ is uniform elliptic, it may happen that $\hat{v}(\lambda,\mu)$ is asymptotically decreasing for $\lambda \rightarrow \infty$. This means that we can not expect to drop the assumption that $\mu$ is large in Theorem \ref{Monotonmularge} even for measures that are bounded away from 0.\\
    \begin{thm}\label{NonIncreasing} Let $d \geq 2$.
        There exists $\mu>0$, $q \in \mathcal{M}_1(\mathbb{R})$ with $q$ uniformly elliptic, such that for $\varepsilon >0$ and $\lambda$ large enough 
        \[\hat{v}(\lambda+\varepsilon,\mu) < \hat{v}(\lambda,\mu).\]
    \end{thm}
    The proof of this theorem will be relying on the asymptotic expression we have for the speed in Corollary \ref{corAssymtoticNVBRW}.
\begin{lem}\label{AlmostPercolation}
        Let $\alpha \in (0,1)$, we define $q = \frac{\delta{\alpha}+\delta_1}{2}$. Then we have that for $m = \frac{1+\alpha}{2}$ and $\mu >0$:
        \[(m+\mu)\frac{\E\left[\frac{1}{(\mu + \tilde{\omega})^2}\right]\E\left[ \frac{\tilde{\omega}}{\mu + \tilde{\omega}}\right] -\E\left[ \frac{\tilde{\omega}}{(\mu + \tilde{\omega})^2}\right] \E\left[\frac{1}{\mu + \tilde{\omega}}\right] }{\E\left[\frac{1}{\mu + \tilde{\omega}}\right]^2} - \frac{\E\left[\frac{\tilde{\omega}}{\mu + \tilde{\omega}}\right]}{\E\left[\frac{1}{\mu + \tilde{\omega}}\right]} = \frac{\alpha^2-(2\mu+6)\alpha + 1 -2\mu}{2(2\mu+1+\alpha)}.\]
    \end{lem}
    \begin{proof}
        We have 
        \begin{eqnarray*}
            &&\E\left[\frac{1}{(\mu + \tilde{\omega})^2}\right]\E\left[ \frac{\tilde{\omega}}{\mu + \tilde{\omega}}\right] -\E\left[ \frac{\tilde{\omega}}{(\mu + \tilde{\omega})^2}\right] \E\left[\frac{1}{\mu + \tilde{\omega}}\right] \\
            &=&\frac{1}{4}\left[\left(\frac{1}{(\mu+\alpha)^2} +\frac{1}{(\mu+1)^2}\right)\left(\frac{\alpha}{\mu+\alpha} +\frac{1}{\mu+1} \right)-\left(\frac{\alpha}{(\mu+\alpha)^2} +\frac{1}{(\mu+1)^2} \right)\left(\frac{1}{\mu+\alpha} +\frac{1}{\mu+1} \right)\right]\\
            &=&\frac{1}{4}\left[\frac{\alpha}{(\mu+\alpha)(\mu+1)^2}+\frac{1}{(\mu+\alpha)^2(\mu+1)}-\frac{\alpha}{(\mu+\alpha)^2(\mu+1)} -\frac{1}{(\mu+\alpha)(\mu+1)^2}\right].\\
        \end{eqnarray*}
        This yields
        \begin{eqnarray*}
            (m+\mu)\frac{\E\left[\frac{1}{(\mu + \tilde{\omega})^2}\right]\E\left[ \frac{\tilde{\omega}}{\mu + \tilde{\omega}}\right] -\E\left[ \frac{\tilde{\omega}}{(\mu + \tilde{\omega})^2}\right] \E\left[\frac{1}{\mu + \tilde{\omega}}\right] }{\E\left[\frac{1}{\mu + \tilde{\omega}}\right]^2} &=& (m+\mu)\frac{\alpha(\mu+\alpha) +\mu +1 -\alpha(\mu+1)-\mu-\alpha}{(2\mu+\alpha+1)^2}\\
            &=&\left(\frac{1+\alpha}{2}+\mu\right) \frac{(1-\alpha)(\mu+1)-(1-\alpha)(\mu+\alpha)}{(2\mu+\alpha+1)^2}\\
            &=& \frac{1}{2} \frac{(1-\alpha)^2}{2\mu+\alpha+1}.
        \end{eqnarray*}
        Further
        \begin{eqnarray*}
            &&\frac{\E\left[\frac{\tilde{\omega}}{\mu + \tilde{\omega}}\right]}{\E\left[\frac{1}{\mu + \tilde{\omega}}\right]} = \frac{2\alpha+\mu(1+\alpha)}{2\mu+\alpha+1}.
        \end{eqnarray*}
        Altogether we get that
        \begin{eqnarray*}
            (m+\mu)\frac{\E\left[\frac{1}{(\mu + \tilde{\omega})^2}\right]\E\left[ \frac{\tilde{\omega}}{\mu + \tilde{\omega}}\right] -\E\left[ \frac{\tilde{\omega}}{(\mu + \tilde{\omega})^2}\right] \E\left[\frac{1}{\mu + \tilde{\omega}}\right] }{\E\left[\frac{1}{\mu + \tilde{\omega}}\right]^2} - \frac{\E\left[\frac{\tilde{\omega}}{\mu + \tilde{\omega}}\right]}{\E\left[\frac{1}{\mu + \tilde{\omega}}\right]}
            &=& \frac{\alpha^2-(2\mu+6)\alpha + 1 -2\mu}{2(2\mu+1+\alpha)}.
        \end{eqnarray*}
    \end{proof}
    \begin{proof}{Theorem \ref{NonIncreasing}\\}
        For $\mu >0$ $\alpha \in (0,1)$ and $q = \frac{\delta_{\alpha}+\delta_1}{2}$ and $\tilde{\omega} \sim q$, we define 
        \[A(\mu,\alpha)=(m+\mu)\frac{\E\left[\frac{1}{(\mu + \tilde{\omega})^2}\right]\E\left[ \frac{\tilde{\omega}}{\mu + \tilde{\omega}}\right] -\E\left[ \frac{\tilde{\omega}}{(\mu + \tilde{\omega})^2}\right] \E\left[\frac{1}{\mu + \tilde{\omega}}\right] }{\E\left[\frac{1}{\mu + \tilde{\omega}}\right]^2} - \frac{\E\left[\frac{\tilde{\omega}}{\mu + \tilde{\omega}}\right]}{\E\left[\frac{1}{\mu + \tilde{\omega}}\right]}. \]
        Then by Lemma \ref{AlmostPercolation} we have that 
        \[A(\mu,\alpha)=\frac{\alpha^2-(2\mu+6)\alpha + 1 -2\mu}{2(2\mu+1+\alpha)}.\]
        So for $\alpha = \mu =0.1$ we have that $A(\mu,\alpha) >0$.\\
        From the asymptotic expression of the speed in Corollary \ref{corAssymtoticNVBRW} we have that 
        \[\hat{v}(\lambda,\mu) = \frac{\E\left[\frac{\tilde{\omega}}{\mu + \tilde{\omega}}\right]}{\E\left[\frac{1}{\mu + \tilde{\omega}}\right]}
        + (2d-2)\frac{\E\left[\frac{\tilde{\omega}(\tilde{\omega} - m)}{(\mu + \tilde{\omega})^2}\right]\E\left[ \frac{1}{\mu + \tilde{\omega}}\right] -\E\left[ \frac{\tilde{\omega}-m}{(\mu + \tilde{\omega})^2}\right] \E\left[\frac{\tilde{\omega}}{\mu + \tilde{\omega}}\right] - \E\left[\frac{\tilde{\omega}}{\mu + \tilde{\omega}}\right] \E\left[\frac{1}{\mu + \tilde{\omega}}\right]}{\E\left[\frac{1}{\mu + \tilde{\omega}}\right]^2} e^{-\lambda} + O(e^{-2\lambda}). \]
        So with the measure $q$ for the environment we get
        \[\hat{v}(\lambda,\mu) = \frac{\E\left[\frac{\tilde{\omega}}{\mu + \tilde{\omega}}\right]}{\E\left[\frac{1}{\mu + \tilde{\omega}}\right]}
        + (2d-2)A(\mu,\alpha)  e^{-\lambda} + O(e^{-2\lambda}). \] 
        Let $C(\mu,\alpha)$ be the implicit constant in the $O(e^{-2\lambda})$, and $\varepsilon >0$.
        \begin{eqnarray*}
            \hat{v}(\lambda + \varepsilon,\mu) -\hat{v}(\lambda,\mu) &\leq& (2d-2)A(\mu,\alpha) \left(e^{-(\lambda+\varepsilon)}-e^{-\lambda}\right) + |C(\mu,\alpha)| (e^{-2\lambda-2\varepsilon}+e^{-2\lambda})\\
            &\leq&  (2d-2)A(\mu,\alpha) \left(e^{-(\lambda+\varepsilon)}-e^{-\lambda}\right) + 2|C(\mu,\alpha)| e^{-2\lambda}.
        \end{eqnarray*}
        So as $A(\mu,\alpha) >0$ we have that for $\lambda$ large enough we get $\hat{v}(\lambda + \varepsilon,\mu) -\hat{v}(\lambda,\mu) <0$.\\
    \end{proof}
\subsection{Monotonicity for $\lambda$ large enough}
The goal of this subsection is to show that for a fixed $\mu$ the speed of the NVBRW $\hat{v}(\lambda,\mu)$ is for $\lambda$ large enough eventually monotone. To do so we will study the asymptotic behavior of the derivative of the speed for $\lambda$ large. We will be using the coupling presented in section \ref{couplingDiffBias}. The computations we will be doing are adapted from the proof of section 4 in \cite{percolation}.
We will start by arguing that $\lambda \mapsto \hat{v}(\lambda,\mu)$ is continuously differentiable.\\
Let $R_{a}$ and $L_{a}$ be the number of attempted jumps respectively in direction $e_1$ and $-e_1$ up to time $\tau_1$, similarly let $R$ and $L$ be the number of succeeded jumps respectively in direction $e_1$ and $-e_1$ up to time $\tau_1$. We will write as before $N$ for the number of attempted jumps up to $\tau_1$ and $(T^g_i)_{i \geq 1}$/$(T^b_i)_{i \geq 1}$/$(T^v_i)_{i \geq 1}$ for respectively good/bad/very bad points.
\begin{lem} \label{lemma3.4Gantert}
    Let $X$ be a NVBRW on dynamical conductances with parameter $\lambda >0$ and $\mu>0$,
    \[\E^{\lambda}[X_{\tau_1}\cdot e_1] = \E^{0}\left[(R-L)e^{R_a-L_a} \left(\frac{2d}{Z_{\lambda}}\right)^N\right].\]
\end{lem}
\begin{proof}
    See the proof of Lemma 3.4 in \cite{percolation} and use the alternative representation of the process.
\end{proof}
\begin{lem}\label{continouslydifferentiable}
    $\lambda \mapsto \hat{v}(\lambda,\mu)$ is continuously differentiable.
\end{lem}
\begin{proof}
    See the proof of Lemma 3.5 in \cite{percolation} and use Lemma \ref{lemma3.4Gantert}.
\end{proof}
\begin{thm}\label{thmDerivative}
    Let $\mu>0$ then there exists a $\lambda_0$ and constants $c,C \in \mathbb{R}$ such that for all $\lambda \geq \lambda_0$
    \[|\frac{d}{d\lambda}\hat{v}(\lambda,\mu) - Ce^{-\lambda}| \leq c e^{-2\lambda}.\]
\end{thm}
We define the following events:
\begin{itemize}
    \item $V = \{T^v_1 \leq \tau_1\} \cap \{T^v_2 > \tau_1\}$ the event that there is exactly one very bad point up to $\tau_1$,
    \item $V_l = V \cap \{T_l = T^v_1\}$ the event that tunique very bad point up to $\tau_1$ is the $l$-th attempted jump,
    \item $G = T^b_1 > \tau_1$ the event that there are no bad points up to time $\tau_1$,
    \item $R = \left\{ \text{at } T^v_1 \text{ the attempted jump is in direction } \pm e_i \text{ for } i \in \{2,...,d\} \right\}$, we will denote the rate of this event by $r_r = \kappa (2d-2)(Z_{\lambda}^{-1}-Z_{\lambda+\varepsilon}^{-1})$.
\end{itemize}
\begin{lem}\label{independentofLambda}
    There exists a $c >0$ such that for $\mu >0$ and for all $k \in \mathbb{N}$ and $l \leq k$
    \[\P(T^b_1 \leq \tau_1|N=k,V_l) \leq (k-1) r_b  \quad \text{ and } \quad \P(R^c|N=k,V_l,T^b_1 > \tau_1) \leq c e^{-\lambda}.\] 
    Moreover there exist functions $f,g:\mathbb{N} \times \mathbb{N} \rightarrow [0,\infty)$ which do not depend on $\lambda$ or on $\varepsilon$, such that 
    \[\E\left[Y_{\tau_1}\cdot e_1 | N=k ,V_l,T^b_1 > \tau_1\right] = f(k,l) \quad \text{ and } \quad \E\left[X_{\tau_1}\cdot e_1 | N=k ,V_l,T^b_1 > \tau_1,R\right] = g(k,l).\]
\end{lem}
\begin{proof}
    We recall that the distribution of the $N$ is independent of the coloring (in good, bad, and very bad points) of the Poisson process $\mathcal{P}$ so conditioned on $N = k$ and $V_l$ for some $l\leq k$ we get that $\forall i \leq k$ with $l \neq i$ we have that $t_i$ has a probability $\frac{r_b}{\kappa}$ of being bad. By union bound we get:
    \[\P(T^b_1 \leq \tau_1|N=k,V_l) \leq (k-1) r_b\] 
    and using that for some $c>0$ $\frac{r_r}{r_v}  \leq c e^{-\lambda}$
    \[\P(R^c|N=k,V_l,T^b_1 > \tau_1) \leq c e^{-\lambda}.\] 
    Further on the event $A = \{N=k\} \cap V_l \cap \{T^b_1 > \tau_1\} \cap R$ we have that for $(U_i)_{i\geq 0}$ the sequence of independent  Uniform$[0,\kappa]$ random variable according to which $X$ and $Y$ decide whether they jump or not :
     \[Y_{\tau_1}\cdot e_1 = \sum_{i = 1}^{k}\mathds{1}_{[0,U_i]}(\omega_{T_i}(Y_{T_i^-},Y_{T_i^-}+e_1))\]
    and
     \[X_{\tau_1}\cdot e_1 = \sum_{i = 1, i\neq l}^{k}\mathds{1}_{[0,U_i]}(\omega_{T_i}(X_{T_i^-},X_{T_i^-}+e_1)).\]
    Further we have that 
    \[\mathcal{L}(T_1,...,T_k,\omega,U_1,...,U_k | N = k,  V_l , T^b_1 > \tau_1, R) =\mathcal{L}(T_1,...,T_k,\omega,U_1,...,U_k | N = k) .\]
    This means that the law of $(T_1,...,T_k,\omega,U_1,...,U_k)$ conditioned on $A$ is independent of $\lambda$. But under $A$ the process $Y$ becomes a walk that only attempts jumps to the right at the times $T_1,...,T_k$, and X attempts jumps to right at times $T_1,...,T_{l-1},T_{l+1},...,T_k$ and at $T_l$ it attempts a jump in one of the $2d-2$ direction. So there are functions $f,g:\mathbb{N} \times \mathbb{N} \rightarrow [0,\infty)$ which do not depend on $\lambda$ or on $\varepsilon$, such that 
    \[\E\left[Y_{\tau_1}\cdot e_1 | N=k ,V_l,T^b_1 > \tau_1\right] = f(k,l)\]
    and
    \[\E\left[X_{\tau_1}\cdot e_1 | N=k ,V_l,T^b_1 > \tau_1,R\right] = g(k,l).\]
\end{proof}
\begin{proof}[Theorem \ref{thmDerivative}]
    We recall from Lemma \ref{regenerativSpeed} and Remark \ref{remarkNormalization} that 
    \[\hat{v}(\lambda+\varepsilon,\mu) - \hat{v}(\lambda,\mu) = \frac{\E\left[Y_{\tau_1}\cdot e_1 - X_{\tau_1}\cdot e_1\right]}{\E\left[\tau_1\right]}.\]
    As 
    \[\E\left[Y_{\tau_1}\cdot e_1 - X_{\tau_1}\cdot e_1| T^v_1 > \tau_1 \right] = 0 . \]
    We have that 
    \[\hat{v}(\lambda+\varepsilon,\mu) - \hat{v}(\lambda,\mu) = \frac{\E\left[Y_{\tau_1}\cdot e_1 - X_{\tau_1}\cdot e_1|T^v_1 \leq \tau_1\right]}{\E\left[\tau_1\right]}\P(T^v_1 \leq \tau_1).\] 
    First we compute 
    \[\P(T^v_1 \leq \tau_1) = \E[\mathds{1}_{T^v_1 \leq \tau_1}] = \E[\E[\mathds{1}_{T^v_1 \leq \tau_1}|N]] = \E\left[1-\left(1-\frac{r_v}{\kappa}\right)^N\right].\]
    Recall that $\frac{r_v}{\kappa} = e^{\lambda+\varepsilon}/Z_{\lambda +\varepsilon} - e^{\lambda}/Z_{\lambda}$.\\
    Since $1-\left(1-\frac{r_v}{\kappa}\right)^N \leq \frac{r_v}{\kappa}N$, by dominated convergence and l'Hôpital we get that 
    \begin{eqnarray*}
        \lim_{\varepsilon \rightarrow 0} \frac{\P(T^v_1 \leq \tau_1)}{\varepsilon} &=& \E[\lim_{\varepsilon \rightarrow 0} N \left(\frac{d}{d\varepsilon}\frac{e^{\lambda+\varepsilon}}{Z_{\lambda +\varepsilon}}\right) (1 - \frac{r_v}{\kappa})^{N-1}]\\
        &=& \E[\lim_{\varepsilon \rightarrow 0} N \frac{e^{\lambda+\varepsilon}Z_{\lambda +\varepsilon} - (e^{\lambda +\varepsilon} - e^{-(\lambda+\varepsilon)}) e^{\lambda+\varepsilon} }{Z_{\lambda +\varepsilon}^2} (1 - \frac{r_v}{\kappa})^{N-1}] \\
        &=& \E[\lim_{\varepsilon \rightarrow 0} N \frac{(2d-2)e^{\lambda+\varepsilon} + 2}{Z_{\lambda +\varepsilon}^2} (1 - \frac{r_v}{\kappa})^{N-1}] \\
        &=& \E[N \frac{(2d-2)e^{\lambda} + 2}{Z_{\lambda}^2}] = (2d-2) e^{-\lambda} \E[N] + O(e^{-2\lambda}).
    \end{eqnarray*}
    Next we want to show that for some constant $C \geq 0$ depending on $q,d,\mu$ we have that
    \[\lim_{\varepsilon \rightarrow 0} \E\left[Y_{\tau_1}\cdot e_1 - X_{\tau_1}\cdot e_1|T^v_1 \leq \tau_1\right] = C + O(e^{-\lambda}).\]
    First we look at 
    \begin{eqnarray*}
        \P(T_2^v > \tau_1|T_1^v \leq \tau_1 ) = \frac{\E\left[N \frac{r_v}{\kappa} \left(1-\frac{r_v}{\kappa}\right)^N\right]}{\E\left[1-\left(1-\frac{r_v}{\kappa}\right)^N\right]}.
    \end{eqnarray*}
    Then using L'Hôpital and then dominated convergence on both the numerator and the denominator yields that 
    \begin{eqnarray} \label{equation1}
        \lim_{\varepsilon \rightarrow 0} \P(T_2^v > \tau_1,|T_1^v \leq \tau_1 ) = 1.
    \end{eqnarray}
    Then we have 
    \begin{eqnarray}\label{equation4.19}
        \E[Y_{\tau_1}\cdot e_1 | T_1^v \leq \tau_1] = &&\E[Y_{\tau_1}\cdot e_1 | T_1^v \leq \tau_1,T_2^v > \tau_1 ]\P(T_2^v > \tau_1|T_1^v \leq \tau_1 ) \\
        &+& \E[Y_{\tau_1}\cdot e_1 |T_2^v \leq \tau_1 ]\P(T_2^v \leq \tau_1|T_1^v \leq \tau_1 ).
    \end{eqnarray}
    As $Y_{\tau_1}\cdot e_1 \leq N$ we get that 
    \begin{eqnarray*}
        \limsup_{\varepsilon \rightarrow 0} \E[Y_{\tau_1}\cdot e_1 |T_2^v \leq \tau_1 ] \leq \lim_{\varepsilon \rightarrow 0} \E[N|T_2^v \leq \tau_1 ] = \lim_{\varepsilon \rightarrow 0} \frac{\E\left[N\left(1-(1-\frac{r_v}{\kappa})^N - N\frac{r_v}{\kappa}(1-\frac{r_v}{\kappa})^{N-1}\right)\right]}{\E\left[1-(1-\frac{r_v}{\kappa})^N - N\frac{r_v}{\kappa}(1-\frac{r_v}{\kappa})^{N-1}\right]}.
    \end{eqnarray*}
    Using again l'Hôpital's rule and dominated convergence 
    \begin{eqnarray*}
        \lim_{\varepsilon \rightarrow 0} \frac{\E\left[N\left(1-(1-\frac{r_v}{\kappa})^N - N\frac{r_v}{\kappa}(1-\frac{r_v}{\kappa})^{N-1}\right)\right]}{\E\left[1-(1-\frac{r_v}{\kappa})^N - N\frac{r_v}{\kappa}(1-\frac{r_v}{\kappa})^{N-1}\right]} = \frac{\E[N^2(N-1)]}{\E[N(N-1)]}.
    \end{eqnarray*}
    Since $\E[N] > 1$ and $N$ has exponential tails by Lemma \ref{finiteTau} and equation \eqref{equation1} we get that 
    \[\lim_{\varepsilon \rightarrow 0}\E[Y_{\tau_1}\cdot e_1 |T_2^v \leq \tau_1 ]\P(T_2^v \leq \tau_1|T_1^v \leq \tau_1 ) = 0 ,\]
    In a similar we get that 
    \begin{eqnarray*}
        \E[X_{\tau_1}\cdot e_1 | T_1^v \leq \tau_1] = &&\E[X_{\tau_1}\cdot e_1 | T_1^v \leq \tau_1,T_2^v > \tau_1 ]\P(T_2^v > \tau_1|T_1^v \leq \tau_1 ) \\
        &+& \E[X_{\tau_1}\cdot e_1 |T_2^v \leq \tau_1 ]\P(T_2^v \leq \tau_1|T_1^v \leq \tau_1 ),
    \end{eqnarray*}
    with 
    \[\lim_{\varepsilon \rightarrow 0}\E[X_{\tau_1}\cdot e_1 |T_2^v \leq \tau_1 ]\P(T_2^v \leq \tau_1|T_1^v \leq \tau_1 ) = 0 .\]
    Next we want to study the first expectation on the right-hand side in equation \eqref{equation1}.
    \begin{eqnarray*}
        \E[Y_{\tau_1}\cdot e_1 | T_1^v \leq \tau_1,T_2^v > \tau_1 ] = \sum_{k = 1}^{\infty} \sum_{l\leq k} \E \left[Y_{\tau_1}\cdot e_1| N = k, V_l \right] \P(N=k,V_l | T_1^v \leq \tau_1,T_2^v > \tau_1 )
    \end{eqnarray*}
    and we split $\E \left[Y_{\tau_1}\cdot e_1| N = k, V_l \right]$ into 
    \begin{eqnarray*}
        &&\E \left[Y_{\tau_1}\cdot e_1| N = k, V_l \right] = \E \left[Y_{\tau_1}\cdot e_1| N = k, V_l, T^g_1 > \tau_1 \right] \\
        &+& \left( \E \left[Y_{\tau_1}\cdot e_1| N = k, V_l, T^g_1 \leq \tau_1 \right] - \E \left[Y_{\tau_1}\cdot e_1| N = k, V_l, T^g_1 > \tau_1 \right]\right) \P(T^g_1 \leq \tau_1 | N=k,V_l).
    \end{eqnarray*}
    Using that $Y_{\tau_1} \cdot e_1 \leq N$ and using Lemma \ref{independentofLambda}
    \begin{eqnarray}
        \E \left[Y_{\tau_1}\cdot e_1| N = k, V_l \right] = f(k,l) + O(k^2 e^{-\lambda}),
    \end{eqnarray}
    so 
    \begin{eqnarray}\label{equation4.22}
        \E[Y_{\tau_1}\cdot e_1 | T_1^v \leq \tau_1,T_2^v > \tau_1 ] = \sum_{k = 1}^{\infty} \sum_{l\leq k} \left(f(k,l) + O(k^2 e^{-\lambda})\right) \P(N=k,V_l | T_1^v \leq \tau_1,T_2^v > \tau_1 ).
    \end{eqnarray}
    Further 
    \begin{eqnarray*}
        \P(N=k,V_l | T_1^v \leq \tau_1,T_2^v > \tau_1 ) = \frac{\P(N=k) \frac{r_v}{\kappa} (1-\frac{r_v}{\kappa})^k}{\E\left[N\frac{r_v}{\kappa} (1-\frac{r_v}{\kappa})^N\right]} = \frac{\P(N=k) (1-\frac{r_v}{\kappa})^k}{\E\left[N(1-\frac{r_v}{\kappa})^N\right]}
    \end{eqnarray*}
    and by dominated convergence 
    \begin{eqnarray} \label{equation4.23}
        \lim_{\varepsilon \rightarrow 0} \P(N=k,V_l | T_1^v \leq \tau_1,T_2^v > \tau_1 ) = \frac{\P(N=k)}{\E[N]}.
    \end{eqnarray}
    Then using that $Y_{\tau_1}\cdot e_1 \leq N $ implies that $f(k,l) \leq k$ gives us if we plug equation \eqref{equation4.22} into equation \eqref{equation4.19} and then use dominated convergence to take the limit with equations \eqref{equation1} \eqref{equation4.23} 
    \begin{eqnarray*}
        \lim_{\varepsilon \rightarrow 0} \E[Y_{\tau_1}\cdot e_1| T^v_1 \leq \tau_1] &=& \sum_{k=1}^{\infty} \sum_{l\leq k} f(k,l) \frac{\P(N=k)}{\E[N]} + O(e^{-\lambda} \sum_{k=1}^{\infty} k^3 \frac{\P(N=k)}{\E[N]} )\\
         &=&\sum_{k=1}^{\infty} \sum_{l\leq k} f(k,l) \frac{\P(N=k)}{\E[N]} + O(e^{-\lambda}) .
    \end{eqnarray*}
    We can similarly get that 
    \begin{eqnarray*}
        \lim_{\varepsilon \rightarrow 0} \E[X_{\tau_1}\cdot e_1| T^v_1 \leq \tau_1] &=& \sum_{k=1}^{\infty} \sum_{l\leq k} g(k,l) \frac{\P(N=k)}{\E[N]} + O(e^{-\lambda} \sum_{k=1}^{\infty} k^3 \frac{\P(N=k)}{\E[N]} ) \\
        &=&\sum_{k=1}^{\infty} \sum_{l\leq k} g(k,l) \frac{\P(N=k)}{\E[N]} + O(e^{-\lambda}) .
    \end{eqnarray*}
    Altogether this gives us that 
    \begin{eqnarray*}
        \lim_{\varepsilon \rightarrow 0}  \frac{\hat{v}(\lambda+\varepsilon,\mu) - \hat{v}(\lambda,\mu)}{\varepsilon} = (2d-2) \frac{\E[N]}{\E[\tau_1]} \sum_{k=1}^{\infty} \sum_{l\leq k} (f(k,l) - g(k,l)) \frac{\P(N=k)}{\E[N]}  e^{-\lambda} + O(e^{-2\lambda}) .
    \end{eqnarray*}
\end{proof}
\begin{thm}\label{assymptoticMonotonicity}
    Let $d \geq 2$, $\mu > 0$ and $m = \E[\tilde{\omega}]$, with $\tilde{\omega} \sim q$. Assume that
    \[(m+\mu)\frac{\E\left[\frac{1}{(\mu + \tilde{\omega})^2}\right]\E\left[ \frac{\tilde{\omega}}{\mu + \tilde{\omega}}\right] -\E\left[ \frac{\tilde{\omega}}{(\mu + \tilde{\omega})^2}\right] \E\left[\frac{1}{\mu + \tilde{\omega}}\right] }{\E\left[\frac{1}{\mu + \tilde{\omega}}\right]^2} - \frac{\E\left[\frac{\tilde{\omega}}{\mu + \tilde{\omega}}\right]}{\E\left[\frac{1}{\mu + \tilde{\omega}}\right]} \neq 0.\] 
    Then there is a $\lambda_0 \geq 0$ such that $\hat{v}(\lambda,\mu)$ is monotonous in $\lambda$ on $(\lambda_0,\infty)$.
    On $(\lambda_0,\infty)$ $\hat{v}(\lambda,\mu)$ is \begin{itemize}
    \item increasing if $(m+\mu)\frac{\E\left[\frac{1}{(\mu + \tilde{\omega})^2}\right]\E\left[ \frac{\tilde{\omega}}{\mu + \tilde{\omega}}\right] -\E\left[ \frac{\tilde{\omega}}{(\mu + \tilde{\omega})^2}\right] \E\left[\frac{1}{\mu + \tilde{\omega}}\right] }{\E\left[\frac{1}{\mu + \tilde{\omega}}\right]^2} - \frac{\E\left[\frac{\tilde{\omega}}{\mu + \tilde{\omega}}\right]}{\E\left[\frac{1}{\mu + \tilde{\omega}}\right]}< 0$,
    \item decreasing if $(m+\mu)\frac{\E\left[\frac{1}{(\mu + \tilde{\omega})^2}\right]\E\left[ \frac{\tilde{\omega}}{\mu + \tilde{\omega}}\right] -\E\left[ \frac{\tilde{\omega}}{(\mu + \tilde{\omega})^2}\right] \E\left[\frac{1}{\mu + \tilde{\omega}}\right] }{\E\left[\frac{1}{\mu + \tilde{\omega}}\right]^2} - \frac{\E\left[\frac{\tilde{\omega}}{\mu + \tilde{\omega}}\right]}{\E\left[\frac{1}{\mu + \tilde{\omega}}\right]} > 0$.
    \end{itemize}
\end{thm}

\begin{proof}
    Let $C$, $c$ be as in the proof of Theorem \ref{thmDerivative}. Then \[\left|\frac{d}{d\lambda}\hat{v}(\lambda,\mu) - Ce^{-\lambda}\right| \leq c e^{-2\lambda}\] and so for $\lambda > 0$ large enough as $\hat{v}(\lambda,\mu)$ is continuously differentiable in $\lambda$ by Lemma \ref{continouslydifferentiable}
    \begin{eqnarray*}
        \hat{v}(2\lambda,\mu)-\hat{v}(\lambda,\mu) = Ce^{-\lambda} + O(e^{-2\lambda}).
    \end{eqnarray*}
    Using Corollary \ref{corAssymtoticNVBRW} we get that \[C = - (2d-2)\left((m+\mu)\frac{\E\left[\frac{1}{(\mu + \tilde{\omega})^2}\right]\E\left[ \frac{\tilde{\omega}}{\mu + \tilde{\omega}}\right] -\E\left[ \frac{\tilde{\omega}}{(\mu + \tilde{\omega})^2}\right] \E\left[\frac{1}{\mu + \tilde{\omega}}\right] }{\E\left[\frac{1}{\mu + \tilde{\omega}}\right]^2} - \frac{\E\left[\frac{\tilde{\omega}}{\mu + \tilde{\omega}}\right]}{\E\left[\frac{1}{\mu + \tilde{\omega}}\right]}\right),\] so if $C\neq 0$ (which is our assumption) we get that $\frac{d}{d\lambda}\hat{v}(\lambda,\mu)$ is eventually of the same sign as $C$ for $\lambda $ large.
\end{proof}
\begin{cor}
    Let $d \geq 2$, $\mu > 0$, then there is a $\lambda_0 \geq 0$ such that $v(\lambda,Z_{\lambda} \mu)$ is increasing on $(\lambda_0,\infty)$. 
\end{cor}
\begin{proof}
    Recall that \[v(\lambda,Z_{\lambda} \mu) = Z_{\lambda} \hat{v}(\lambda,\mu).\] As $Z_{\lambda}$ is continuously differentiable in $\lambda$ we get that \[\frac{d}{d\lambda}v(\lambda,Z_{\lambda}\mu) = Z_{\lambda}\frac{d}{d\lambda}\hat{v}(\lambda,\mu) + (e^{\lambda} - e^{-\lambda})\hat{v}(\lambda,\mu),\] which yields for $\lambda$ large enough
    \begin{eqnarray*}
        \left|\frac{d}{d\lambda}v(\lambda,Z_{\lambda}\mu) - e^{\lambda} \hat{v}(\lambda,\mu)\right| \leq e^{\lambda} \frac{d}{d\lambda}\hat{v}(\lambda,\mu) + O(e^{-\lambda}).
    \end{eqnarray*}
    Further using Corollary \ref{corAssymtoticNVBRW} and Theorem \ref{thmDerivative} we get that there is a constant $\tilde{C}$ such that for $\lambda$ large enough and $\tilde{\omega}\sim q$
    \begin{eqnarray*}
        \left|\frac{d}{d\lambda}v(\lambda,Z_{\lambda}\mu) - e^{\lambda} \frac{\E\left[\frac{\tilde{\omega}}{\mu + \tilde{\omega}}\right]}{\E\left[\frac{1}{\mu + \tilde{\omega}}\right]}\right| \leq \tilde{C} + O(e^{-\lambda}).
    \end{eqnarray*}
    So for $\lambda$ large enough \[\frac{d}{d\lambda}v(\lambda,Z_{\lambda}\mu) > 0\] and the speed is increasing.
\end{proof} 

\section{Open questions}
We here want to give some open questions related to this work.
\begin{itemize}
    \item How does the process behave in the critical case when \[\E\left[\frac{\tilde{\omega}(\tilde{\omega} - m)}{(\mu + \tilde{\omega})^2}\right]\E\left[ \frac{1}{\mu + \tilde{\omega}}\right] -\E\left[ \frac{\tilde{\omega}-m}{(\mu + \tilde{\omega})^2}\right] \E\left[\frac{\tilde{\omega}}{\mu + \tilde{\omega}}\right] - \E\left[\frac{\tilde{\omega}}{\mu + \tilde{\omega}}\right] \E\left[\frac{1}{\mu + \tilde{\omega}}\right] = 0\]
    This question has been studied for the dynamical percolation in \cite{olzhabayev2025biasedrandomwalkcritical}.
    \item In the regime where the speed is asymptotically monotone increasing, so\[\E\left[\frac{\tilde{\omega}(\tilde{\omega} - m)}{(\mu + \tilde{\omega})^2}\right]\E\left[ \frac{1}{\mu + \tilde{\omega}}\right] -\E\left[ \frac{\tilde{\omega}-m}{(\mu + \tilde{\omega})^2}\right] \E\left[\frac{\tilde{\omega}}{\mu + \tilde{\omega}}\right] - \E\left[\frac{\tilde{\omega}}{\mu + \tilde{\omega}}\right] \E\left[\frac{1}{\mu + \tilde{\omega}}\right] < 0,\] 
    is the speed monotone increasing on $[0,+\infty)$? This question has been studied in static uniform elliptic conductances in \cite{staticConductancesMonotonicity}.
    \item If we fix the bias $\lambda$ and take $\mu \rightarrow 0$ do we recover for uniformly elliptic environments the speed on static conductances?
\end{itemize}

\section{Appendix}
To prove Lemma \ref{deathbirthchain}, we first need a lemma about hitting time of random walks.
\begin{proof}[Proof of Lemma \ref{deathbirthchain}]\label{deathbirthchainproof}
    As the transition rates  of $(A_t)_{t \geq 0}$ are bounded away from zero we can look at the embedded discrete Markov chain $(A_n)_{n \geq 0}$ instead of the continuous time Markov process.
    
    Let $x =  \lceil\alpha \mu^{-1} \rceil + 1$, $T^x_1 = \inf\{n \geq 1: A_n = x\}$ and $T^x_{k+1} = \inf\{n >T^x_k: A_n = x\}$ for $k \geq 1$. For easier notation we will write $T^x_0 = 0 $. Further we define $N = \max\{k \geq 0 : T^x_k < \tau_1\}$.\\
    We now can write 
    \begin{eqnarray*}
        \tau_1 = \sum_{k = 0}^{N-1} \left(T^x_{k+1} - T^x_{k}\right) + \left(\tau_1 - T^x_N\right).
    \end{eqnarray*}
    Note that for $\left(T^x_{k+1} - T^x_{k}\right)_{k \geq 1}$ are i.i.d.\\
    First we see that $\left(\tau_1 - T^x_N\right)$ can be expressed as a hitting time (from $x$ to $0$ if $N \geq 1$ or from $0$ to $0$ if $N=0$) on a finite irreducible Markov chain. So there exists an $\varepsilon_1$ such that for all $\varepsilon < \varepsilon_1$,\; $\E[e^{\varepsilon\left(\tau_1 - T^x_N\right)}] < \infty$.\\
    The same way we also get that there exists $\varepsilon_2$ such that for all $\varepsilon < \varepsilon_2$,\; $\E[e^{\varepsilon T^x_1}] < \infty$.\\
    Next, we see that on the event $\{A_{T^x_k +1}  < x\}$, $T^{\geq x}_{k} -T^{ x}_{k}$ can be seen as a hitting time on a finite irreducible Markov chain, so there is an $\varepsilon_3$ such that for all $\varepsilon < \varepsilon_3$,\; $\E[e^{\varepsilon\left(T^{x}_{k} -T^{ x}_{k}\right)} \mathds{1}_{A_{T^x_k +1}  < x}] < \infty$.\\
    Next we look at $T^{ x}_{k+1}-T^{x}_{k}$ on the event $\{A_{T^x_k +1}  > x\}$ . Let $(X_n)_{n\geq 0}$ be a random walk with $X_0 = x$ a.s. and i.i.d. increments $(\zeta_i)_{i \geq 0}$ with \[\P(\zeta_1 = 1) = 1- \P(\zeta_1 = -1) = \frac{\alpha}{\alpha + x\mu} < 1/2.\]
    We can ten couple $X$ and $A$ such that for all $n \in \{0,...,T^{ x}_{k+1}-T^{x}_{k}\}$ we have that $A_{T^x_{k}+n} \leq X_n$. This means that for $H = \inf\{n \geq 0: X_n-X_0 = -1\}$ we have $T^{ x}_{k+1}-T^{x}_{k} \leq H$. Then using that $H$ has an exponential tail we get there exists an $\varepsilon_4$ such that for all $\varepsilon < \varepsilon_4$,\; $\E[e^{\varepsilon\left(T^{ x}_{k+1}-T^{x}_{k}\right)} \mathds{1}_{A_{T^x_k +1}  > x}] < \infty$.\\
    So for $\varepsilon < \min\{\varepsilon_2, \varepsilon_3,\varepsilon_4\}$ and all $k \geq 0$, we have $\E[e^{\varepsilon\left(T^{ x}_{k+1}-T^{x}_{k}\right)}] < \infty$, and by dominated convergence we get that $\lim_{\varepsilon \rightarrow 0 } \E[e^{\varepsilon\left(T^{ x}_{k+1}-T^{x}_{k}\right)}] = 1$.\\
    Further let $p = \prod_{k=1}^{x} \frac{\mu k}{\mu k +\alpha}$ then $\P(N = k) \leq \P(N \geq k) \leq (1-p)^k$. Then we can choose $\varepsilon > 0$ small enough such that using Cauchy schwarz inequaltiy we get for some $\rho <1$,
    \begin{eqnarray}
        \E[\prod_{k = 0}^{n-1} e^{\varepsilon\left(T^{ x}_{k+1}-T^{x}_{k}\right)} \mathds{1}_{N = n}] \leq \rho^{\frac{n}{2}}.
    \end{eqnarray}
    Then using dominated convergeance one gets
    \begin{eqnarray*}
        \E[\prod_{k = 0}^{N-1} e^{\varepsilon\left(T^{ x}_{k+1}-T^{x}_{k}\right)}] &=& \sum_{n = 0}^{\infty} \E[\prod_{k = 0}^{N-1} e^{\varepsilon\left(T^{ x}_{k+1}-T^{x}_{k}\right)}\mathds{1}_{N=n}] < \infty .\\
    \end{eqnarray*}
    Altogether we have that for $\tilde{\varepsilon} < \min\{\varepsilon, \varepsilon_1\}$, 
    \[\E[e^{\varepsilon \tau_1}] = \E[e^{\varepsilon\left(\tau_1 - T^x_N\right)}] \E[\prod_{k = 0}^{N-1} e^{\varepsilon\left(T^{ x}_{k+1}-T^{x}_{k}\right)}] < \infty.\]
\end{proof}

\textbf{Acknowledgements}\\
I am grateful to my supervisor, Nina Gantert, for her advice and feedback. I am also thankful to Carlo Scali for his helpful advice

\printbibliography
\end{document}